\documentclass[11pt,twoside,a4paper]{article}
\usepackage[english]{babel}
\usepackage{a4wide}
\usepackage{graphicx}
\usepackage{amsmath,amssymb,amsthm, amsgen}
\usepackage{listings} 
\usepackage{color}
\usepackage[usenames,dvipsnames]{xcolor}
\usepackage{rotating}
\usepackage{hhline}
\usepackage{multirow}
\usepackage{enumerate}
\usepackage[bookmarks]{}
\usepackage{amsfonts}   
\usepackage{dsfont}
\usepackage{tikz} 
\usepackage{booktabs}
\usepackage{longtable}
\usepackage{stmaryrd} 
\usepackage{url}
\usepackage{environ}

\newtheorem{thm}{Theorem}[section]
\newtheorem{prop}[thm]{Proposition}
\newtheorem{lem}[thm]{Lemma}

\newtheorem{defn}[thm]{Definition}

\newtheorem{assn}[thm]{Assumption}
\theoremstyle{definition}

\newtheorem{ex}[thm]{Example}

\newcommand{\lrang}[1]{\left\langle {#1} \right\rangle}

\newcommand{\weakto}{ \rightharpoonup }

\newcommand{\E}{\mathbb E}
\newcommand{\e}{\varepsilon}
\newcommand{\N}{\mathbb N}

\renewcommand{\P}{\mathbb P}
\newcommand{\Q}{\mathbb Q}
\newcommand{\R}{\mathbb R}
\newcommand{\T}{\mathbb T}

\newcommand{\Z}{\mathbb Z}

\newcommand{\bone}{\mathbf 1}

\newcommand{\cC}{\mathcal C}
\newcommand{\cD}{\mathcal D}
\newcommand{\cE}{\mathcal E}
\newcommand{\cF}{\mathcal F}
\newcommand{\cG}{\mathcal G}

\newcommand{\cK}{\mathcal K}
\newcommand{\cL}{\mathcal L}
\newcommand{\cM}{\mathcal M}

\newcommand{\cP}{\mathcal P}

\newcommand{\cS}{\mathcal S}
\newcommand{\cT}{\mathcal T}

\newcommand{\cW}{\mathcal W}

\newcommand{\ov}[1]{\overline{#1}}

\NewEnviron{intermezzo}{
    \par\bigskip 
    \addtolength{\leftskip}{10mm} 
    \addtolength{\rightskip}{10mm} 
    \footnotesize 
    \noindent 
    \BODY 
    \par\bigskip 
}

\DeclareFontFamily{U}{mathx}{\hyphenchar\font45}
\DeclareFontShape{U}{mathx}{m}{n}{
      <5> <6> <7> <8> <9> <10>
      <10.95> <12> <14.4> <17.28> <20.74> <24.88>
      mathx10
      }{}
\DeclareSymbolFont{mathx}{U}{mathx}{m}{n}
\DeclareFontSubstitution{U}{mathx}{m}{n}
\DeclareMathAccent{\widecheck}{0}{mathx}{"71}

\definecolor{darkgreen}{rgb}{0,0.7,0}
\definecolor{orange}{rgb}{1,0.45,0}
\long\def\comm#1{{\color{orange}}} 

\allowdisplaybreaks

\title{Hydrodynamic limit for Glauber-Kawasaki dynamics\\
on the Sierpi\'nski gasket}
\author{Patrick van Meurs$\,^{1)}$, Kenkichi Tsunoda$\,^{2)}$}
\date{} 
 
\begin{document}

\maketitle


\begin{abstract}
We prove the hydrodynamic limit for Glauber-Kawasaki dynamics on the Sierpi\'nski gasket, a prototypical fractal graph that lacks translational invariance. The main novelty lies in incorporating Glauber dynamics, allowing for particle creation and annihilation with birth-death rates depending locally on the particle configuration. In the macroscopic limit, the particle density evolves according to a nonlinear reaction--diffusion equation, where the reaction term is explicitly determined by the microscopic rates. The key new ingredient is a replacement lemma adapted to the fractal geometry of the Sierpi\'nski gasket. We establish this lemma by deriving 1-block and 2-blocks estimates on the Sierpi\'nski gasket graph, which require new arguments due to the absence of classical lattice structures.

\footnote{
\hskip -6mm 
${}^{1)}$ Faculty of Mathematics and Physics, Kanazawa University, Kakuma, Kanazawa 920-1192, Japan.
e-mail: pjpvmeurs@staff.kanazawa-u.ac.jp \\
${}^{2)}$ Faculty of Mathematics, Kyushu University, 744 Motooka, Nishi-ku, Fukuoka, 819-0395, Japan. e-mail: tsunoda@math.kyushu-u.ac.jp}
\footnote{
\hskip -6mm
\textbf{Keywords}: Interacting particle system, hydrodynamic limit, Sierpi\'nski gasket.}
\footnote{
\hskip -6mm
\textbf{2020MSC}: 60K35, 82C22, 28A80. 
}
\end{abstract}

\tableofcontents

\section{Introduction}
\label{s:intro}

We aim to contribute to the study of the universality of hydrodynamic limits for interacting particle systems with respect to the underlying discrete space. Classically, this space is a lattice, such as the Euclidean lattice or a discrete torus. Lattices possess translation invariance, homogeneous volume growth, and simple local geometry. These properties play a crucial role in classical techniques for proving hydrodynamic limits, such as entropy methods, replacement lemmas, and block estimates \cite[Chapter 5]{KL}.

However, many physical and biological systems evolve in heterogeneous or irregular media. Modelling such media by a lattice is a crude approximation. Indeed, porous materials, polymer networks, dendritic structures, and certain biological tissues exhibit geometric complexity incompatible with Euclidean homogeneity. In such environments, several analytic properties change: volume growth may be irregular, diffusion may become anomalous, and resistance rather than Euclidean distance becomes the natural metric.

Only recently, several works on hydrodynamic limits beyond lattices have been established. For instance, \cite{Tan} studies the hydrodynamic limit of exclusion processes on a crystal lattice; however, the macroscopic space is still a continuum torus. \cite{vGR} studies the hydrodynamic limit on a compact Riemannian manifold. Recently, a generalization to non-compact Riemannian manifolds was established in \cite{JRV}.

In addition, \cite{Jar,CG} establish hydrodynamic limits on a specific fractal: the Sierpi\'nski gasket (see Figure \ref{fig:Sier} below). The gasket is a canonical example of a self-similar set supporting a well-developed analytic structure, including a resistance form, a Laplacian, and sub-Gaussian heat kernel estimates. At the same time, it lacks translation invariance and has bottlenecks (also called hot spots). To handle these features, \cite{Jar,CG} have built new techniques to prove hydrodynamics limits.

Our aim is to build further on these two works. We briefly review them. In \cite{Jar}, the zero-range process with Glauber dynamics at the three boundary points of the Sierpi\'nski gasket is considered. The Glauber dynamics can be interpreted as reservoirs attached to each boundary site, each having a fixed particle density. The main result of \cite{Jar} is the hydrodynamic limit of the particle system, in which the particle density satisfies a nonlinear heat equation with Dirichlet boundary conditions determined by the reservoir densities. The proof relies on the $H^{-1}$-method, based on sharp asymptotics of the Green function of the gasket.

Second, \cite{CG} considers the symmetric simple exclusion process (Kawasaki dynamics) with similar Glauber dynamics at the boundary. An important generalization is that the scaling of the Glauber rates is regulated by a parameter $b > 0$; for larger values of $b$, the Glauber dynamics becomes less effective. One of the main results of \cite{CG} is the hydrodynamic limit, in which the particle density satisfies the heat equation. The boundary condition depends on $b$: for small $b$ it is of Dirichlet type, for large $b$ it is of Neumann type, and for the critical value $b = \frac{5}{3}$ it is of Robin type. The proof relies on the classical approach of passing to the limit in the associated martingale problems. However, the classical version of the \textit{moving particle lemma} does not apply. On a lattice, this lemma allows one to decompose long-range particle exchanges into nearest-neighbor swaps along a shortest path. Instead, \cite{Chen2} constructs a new moving particle lemma on the gasket, whose proof uses the self-similar structure of the gasket. Using this result, \cite{CG} establishes the hydrodynamic limit described above. They also study dynamical fluctuations around a fixed particle density when the boundary densities are equal.

The main result of the present paper (Theorems \ref{t} and \ref{t:R}) extends the hydrodynamic limit in \cite{CG} to the case where Glauber dynamics is added at each site. We allow for a class of Glauber rates that may depend locally on the particle configuration, thereby enriching the interactions between particles. In the hydrodynamic limit, the Glauber term manifests itself as a nonlinear reaction term in the heat equation. The boundary condition remains the same as in \cite{CG}. 

A notable difference between our result and that of \cite{CG} is that the limiting equation is nonlinear, which is in contrast with the linear equation in \cite{CG}. The presence of the nonlinear reaction term makes the proof much more delicate; it requires a new so-called \textit{replacement lemma}. We prove it by establishing classical 1-block and 2-blocks estimates, which appear to be new on the Sierpi\'nski gasket. Indeed, inspired by the results of \cite{JLS,JLS2}, a local version of the 1-block estimate was established in \cite{Jar}, but no 2-blocks estimate was established. Another related study appears in the preprint \cite{Chen}. 

One of the key ingredients in our proof of the block estimates is the moving particle lemma established in \cite{CG}. In addition, as mentioned above, the lack of translation invariance prevents us from using the classical approach to establish block estimates. To overcome this difficulty, we instead exploit the self-similar structure of the Sierpi\'nski gasket.

With the 2-blocks estimate established in this paper, we pave the way for studying the corresponding fluctuation and large deviation problems. Although fluctuations for this model have been studied in \cite{CG}, the corresponding replacement lemma (often referred to as the Boltzmann-Gibbs principle) was not needed at the fluctuation level since their hydrodynamic equation is linear in the bulk. Our proof of the block estimates is intended as a prototype for continuing the study of fluctuations; see, for example, \cite[Chapter 10]{KL}. Indeed, in the Euclidean setting, it is well known that replacement lemmas play a crucial role in establishing dynamical large deviation principles from hydrodynamic limits \cite{KOV}. For the Glauber-Kawasaki dynamics, dynamical large deviation principles are studied in \cite{JLV,LT}, as well as static large deviation principles in \cite{FLT}. We leave the extension of these results to the case of the Sierpi\'nski gasket for future work.

The paper is organized as follows. In Section \ref{sec:model}, we introduce our model and state the main results precisely for the case $b < \frac53$. A self-contained introduction to the Sierpi\'nski gasket and its discretization is also provided. In Section \ref{s:pf}, we give a sketch of the proof of our main result. Since the overall strategy is similar to that in \cite{CG}, except for the replacement lemma, we highlight only the necessary modifications. In Section \ref{s:repl}, we present the proof of the replacement lemma in full detail. In Section \ref{s:b}, we establish the hydrodynamic limit in the remaining case $b \geq \frac{5}{3}$.

\section{Model and main result}\label{sec:model}

In this section, we introduce the particle system, the reaction--diffusion equation, and finally the hydrodynamic limit (Theorem \ref{t}), which connects the two. Theorem \ref{t} covers one part of the parameter regime; Theorem \ref{t:R} covers the remaining part.

\subsection{The Sierpi\'nski gasket $K$ and its discretization}
\label{s:SG}

In this section, we introduce the Sierpi\'nski gasket $K$ and a discretization of it given in terms of a graph $\cG_N = (V_N, E_N)$. Figure \ref{fig:Sier} illustrates the corresponding definitions. We refer to the textbooks \cite{Kig,Str} for a more complete treatment of fractals and the Sierpi\'nski gasket.

\begin{figure}[h!]
  \centering
  \includegraphics{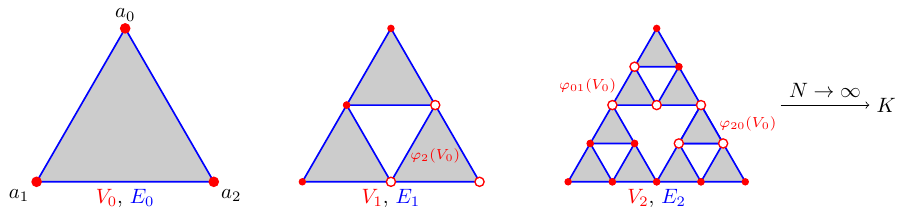} 
  \caption{Discretizations of the Sierpi\'nski gasket at depths $N=0,1,2$. $\varphi_w(V_0)$ is highlighted for several words $w$. }
 \label{fig:Sier}
\end{figure}

Let $V_0 = \{a_0, a_1, a_2\} \subset \R^2$ be the 3 boundary sites of the equilateral triangle given by
\[
  a_0 = \begin{bmatrix} \cos ( \tfrac\pi3 ) \\ \sin ( \tfrac\pi3 ) \end{bmatrix} , \qquad
  a_1 = \begin{bmatrix} 0 \\ 0 \end{bmatrix}, \qquad
  a_2 = \begin{bmatrix} 1 \\ 0 \end{bmatrix}.
\]
Let 
\[
  \varphi_i : \R^2 \to \R^2, \qquad
  \varphi_i(x) := \frac{x + a_i}2 \qquad
  \text{for } i = 0,1,2
\]
be the three corresponding contractive similitudes. The Sierpi\'nski gasket $K \subset \R^2$ is the unique solution  of
\[
  K = \bigcup_{i=0}^2 \varphi_i(K)
\]
over all compact, nonempty subsets of $\R^2$.

To identify self-similar subcomponents of $K$, we introduce the alphabet $\{0,1,2\}$. We denote words by $w = w_1 w_2 \cdots w_j$, where $j \ge 0$, $w_i \in \{0,1,2\}$ are letters and $|w| := j$ is the length. The only difference between words and the ternary numbers is that words may start with a string of zeroes. We set $\varphi_w := \varphi_{w_1} \circ \varphi_{w_2} \circ \cdots \circ \varphi_{w_j}$. Then, $K_w := \varphi_w(K) \subset K$ is a miniature (self-similar) version of $K$. We call $K_w$ a $j$-cell if $|w| = j$.

Next, we construct the sequence of graphs $\cG_N = (V_N, E_N)$ indexed over $N \in \N$, which are increasingly finer discretizations of $K$. Let 
\[
  V_N := \bigcup_{|w| = N} \varphi_w(V_0),
\]
and $E_N$ be all pairs of sites $x,y \in V_N$ (denoted by $x \sim y$) for which there exists a word $w$ with $|w| = N$ for which $x,y \in \varphi_w(V_0)$. It follows that
\begin{equation} \label{VN:EN:abs}
  |V_N| = \frac32 (3^N+1), \qquad
  |E_N| = 3^{N+1}.
\end{equation}
We introduce 
\begin{equation*}
  V_N^0 := V_N \setminus V_0
\end{equation*}
for the set of sites without the boundary sites.

Let $\delta_x$ be the Dirac measure at $x\in\R^2$ and 
\begin{equation*}
  m_N := \frac1{|V_N|} \sum_{x \in V_N} \delta_x
\end{equation*}
the probability measure of the uniform distribution over $V_N$. It converges weakly to $m$ as $N \to \infty$, where the limit measure $m$ is the self-similar probability measure on $K$, which equals $d_H := \frac{\log 3}{\log 2}$ times the $d_H$-dimensional Hausdorff measure restricted to $K$. We fix $(K,m)$ as the measure space related to $K$.

\subsection{Glauber-Kawasaki dynamics on $\cG_N$}
\label{s:GK}

The Glauber-Kawasaki dynamics that we consider is a Markov process $\{\eta_t^N : t \geq 0\}$ on $\Omega_N := \{0,1\}^{V_N}$. It is generated by the sped up operator $5^N \cL_N$ acting on funtions $\varphi:\Omega_N\to\R$ as
\begin{align*}
\cL_N \varphi &:= \cL_N^K\varphi + 5^{-N} \cL_N^G\varphi + b^{-N} \cL_N^B\varphi, \\
  \cL_N^K \varphi (\eta) &:= \sum_{x \in V_N} \sum_{\substack{ y \in V_N \\ y \sim x }} \eta(x)[1 - \eta(y)] [\varphi(\eta^{xy}) - \varphi(\eta)], \\
  \cL_N^G \varphi (\eta) &:= \sum_{x \in V_N^0} c_x ( \eta ) [\varphi(\eta^x) - \varphi(\eta)], \\
  \cL_N^B \varphi (\eta) &:= \sum_{a \in V_0} [ \lambda_-(a) \eta(a) + \lambda_+(a) (1-\eta(a)) ] [\varphi(\eta^a) - \varphi(\eta)],
\end{align*}
where, as usual,
\begin{equation} \label{eta:xy} 
   \eta^{xy}(z)
   := \left\{ \begin{aligned}
     &\eta(y)
     &&\text{if } z=x \\
     &\eta(x)
     &&\text{if } z=y \\
     &\eta(z)
     &&\text{otherwise }
   \end{aligned} \right.
\end{equation} 
is the configuration obtained from $\eta$ by swapping the occupation variables at $x$ and $y$, and
\begin{equation*}
   \eta^x(z)
   := \left\{ \begin{aligned}
     &1-\eta(z)
     &&\text{if } z=x \\
     &\eta(z)
     &&\text{otherwise }
   \end{aligned} \right.
\end{equation*} 
is the configuration obtained from $\eta$ by flipping the occupation variable at $x$. Moreover, $b > 0$, $\lambda_\pm : V_0 \to (0,\infty)$ and $c_x : \Omega_N \to (0,\infty)$ for all $x \in V_N$ are given parameters. 

The generator $\cL_N^K$ describes the Kawasaki dynamics; for each edge $x \sim y$, the occupation variables $\eta(x)$ and $\eta(y)$ swap at unit rate. Next, $\cL_N^G$ describes the Glauber dynamics; at each site $x \in V_N^0$, $\eta(x)$ flips (i.e.\ turns into $1-\eta(x)$) at rate $c_x(\eta)$. The rate $c_x(\eta)$ may depend on $\eta$ at sites other than $x$, and can therefore describe particle interactions. We describe the assumptions on $c_x(\eta)$ below in Assumption \ref{a:cx}. Finally, $\cL_N^B$ describes the Glauber dynamics at the three boundary points $a \in V_0$. This could technically be unified with $\cL_N^G$, but we keep them separate because:
\begin{enumerate}
  \item their scaling in $\cL_N$ is different; it is regulated by $b > 0$, and
  \item at $a \in V_0$ the Glauber rates are simpler; they depend on $\eta$ only through $\eta(a)$. The corresponding birth-death rates $\lambda_\pm(a) > 0$ are commonly interpreted as jumps to and from reservoirs; one for each $a \in V_0$. 
\end{enumerate}

Next we focus on $c_x (\eta)$. Formally speaking, we assume that it is positive, local as a function of $\eta$, and self-similar in the sense that, rather than on $x$, it only depends on the local, self-similar structure of $\cG_N$ around $x$. Next, we introduce these assumptions rigorously. 

For $x,y \in V_N$, let $L(x,y) \in \N$ be the length of the shortest path that connects $x$ and $y$ on $\cG_N$. Take $L_0 \geq 1$ as a fixed integer. Let
 \begin{equation} \label{Lambda:x}
    \Lambda_x 
    := \{ y \in V_N : L(x,y) \leq L_0\} 
    \quad
    \text{for } x \in V_N^0
  \end{equation}
be a discrete neighborhood of $x$ on $V_N$. Note that $|\Lambda_x|$ is uniformly bounded in $N$. We define the translated and rescaled version 
\begin{equation*}
  2^N(\Lambda_x - x) :=   \{ 2^N(y -x) : y \in \Lambda_x \} \subset \R^2,
\end{equation*}  
which is anchored at the origin. The rescaling is such that neighboring sites are separated by a unit distance. For $x,y \in V_N$ we define
\begin{equation*}
  \Lambda_x \sim \Lambda_y
  \ :\Longleftrightarrow \
  \Lambda_x - x = \Lambda_y - y.
\end{equation*}
Note that $\Lambda_x \sim \Lambda_y$ if and only if they have the same shape and rotation. By the self-similar structure of $V_N$, the set
\begin{equation*}
   \cS := \{ 2^N(\Lambda_x - x) : x \in V_N^0 \}
 \end{equation*} 
is independent of $N$ for $N$ large enough with respect to $L_0$. Figure \ref{fig:S} illustrates for small $L_0$ all elements $\Lambda \in \cS$ of different shapes. The rigorous description of our assumptions on $c_x$ is as follows.

\begin{assn}\label{a:cx}
There exists $L_0 \geq 1$ and a family of local functions $c(\cdot; \Lambda) : \{0,1\}^\Lambda \to (0,\infty)$ parametrized by $\Lambda \in \cS$ such that for all $N$ large enough, all $x \in V_N^0$ and all $\eta \in \Omega_N$
  \begin{equation} \label{cx}
    c_x(\eta) = c \big( \eta|_{\Lambda_x}; 2^N(\Lambda_x - x) \big).
  \end{equation} 
\end{assn}

Four remarks are in order. 
First, in \eqref{cx} we abuse notation, since the domain of $\eta|_{\Lambda_x}$ is $\Lambda_x$ whereas it should be $2^N(\Lambda_x - x)$. However, as graphs, they are equal; the only difference is the spatial location of the sites, which are linearly mapped to one another by translation and scaling. Moreover, $\eta|_{\Lambda_x}$ is a vector of $0$'s and $1$'s, and thus changing the domain in this manner amounts merely to a change of notation for the index.
Second, $c_x(\eta)$ depends on $N$, but we do not adopt this in the notation.  
Third, $\{ c(\xi; \Lambda) : \Lambda \in \cS, \ \xi \in \{0,1\}^\Lambda \}$ is a finite set of positive numbers independent of $N$ for all $N$ large enough. Hence, $c_x(\eta) > 0$ and  
\begin{equation} \label{cinf}
  |c_x(\eta)|
  \leq \max_{\Lambda \in \cS} \max_{\xi \in \{0,1\}^\Lambda} c(\xi; \Lambda)
  =: \|c\|_\infty < \infty,
\end{equation}
both uniformly in $x,\eta,N$. 
Fourth, if $L(x,a) \leq L_0 - 1$ for some $a \in V_0$, then $\Lambda_x \not\sim \Lambda_y$ for all $y \in V_N^0 \setminus \{x\}$. The set of all such sites $x$ is finite uniformly in $N$ and does therefore not play a significant role. 

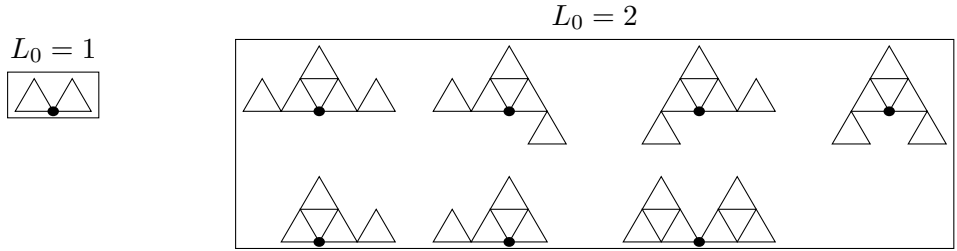
\begin{figure}[h]
\centering
\begin{tikzpicture}[scale=.5, yscale=.866]
    \def \r {0.15}
    \def \lgray {black!20!white}
    
    \draw (-1.2,-.2) rectangle (1.2,1.2);
    \draw (0,1.2) node[above] {$L_0 = 1$};
    \fill (0,0) circle (\r);
    \draw (0,0) -- (1,0) -- (.5,1) -- cycle;
    \begin{scope}[shift={(-1,0)}]
      \draw (0,0) -- (1,0) -- (.5,1) -- cycle;
    \end{scope} 
    
    \begin{scope}[shift={(6,0)}]
    
    \draw (-1.2,-4.2) rectangle (17.7,2.2);
    \draw (8.25,2.2) node[above] {$L_0 = 2$};
    
    \fill (1,0) circle (\r);
    \begin{scope}[scale = 2]
      \draw (0,0) -- (1,0) -- (.5,1) -- cycle;
    \end{scope} 
    \begin{scope}[shift={(.5,1)}]
      \draw (0,0) -- (1,0) -- (.5,-1) -- cycle;
    \end{scope}
    \begin{scope}[shift={(-1,0)}]
      \draw (0,0) -- (1,0) -- (.5,1) -- cycle;
    \end{scope}  
    \begin{scope}[shift={(2,0)}]
      \draw (0,0) -- (1,0) -- (.5,1) -- cycle;
    \end{scope}   
    
    \begin{scope}[shift={(5,0)}]
    \fill (1,0) circle (\r);
    \begin{scope}[scale = 2]
      \draw (0,0) -- (1,0) -- (.5,1) -- cycle;
    \end{scope} 
    \begin{scope}[shift={(.5,1)}]
      \draw (0,0) -- (1,0) -- (.5,-1) -- cycle;
    \end{scope}
    \begin{scope}[shift={(-1,0)}]
      \draw (0,0) -- (1,0) -- (.5,1) -- cycle;
    \end{scope}  
    \begin{scope}[shift={(1.5,-1)}]
      \draw (0,0) -- (1,0) -- (.5,1) -- cycle;
    \end{scope} 
    \end{scope}
    
    \begin{scope}[shift={(10,0)}]
    \fill (1,0) circle (\r);
    \begin{scope}[scale = 2]
      \draw (0,0) -- (1,0) -- (.5,1) -- cycle;
    \end{scope} 
    \begin{scope}[shift={(.5,1)}]
      \draw (0,0) -- (1,0) -- (.5,-1) -- cycle;
    \end{scope}
    \begin{scope}[shift={(2,0)}]
      \draw (0,0) -- (1,0) -- (.5,1) -- cycle;
    \end{scope}  
    \begin{scope}[shift={(-.5,-1)}]
      \draw (0,0) -- (1,0) -- (.5,1) -- cycle;
    \end{scope} 
    \end{scope}
    
    \begin{scope}[shift={(15,0)}]
    \fill (1,0) circle (\r);
    \begin{scope}[scale = 2]
      \draw (0,0) -- (1,0) -- (.5,1) -- cycle;
    \end{scope} 
    \begin{scope}[shift={(.5,1)}]
      \draw (0,0) -- (1,0) -- (.5,-1) -- cycle;
    \end{scope}
    \begin{scope}[shift={(-.5,-1)}]
      \draw (0,0) -- (1,0) -- (.5,1) -- cycle;
    \end{scope}  
    \begin{scope}[shift={(1.5,-1)}]
      \draw (0,0) -- (1,0) -- (.5,1) -- cycle;
    \end{scope} 
    \end{scope}
    
    \begin{scope}[shift={(0,-4)}]
    \fill (1,0) circle (\r);
    \begin{scope}[scale = 2]
      \draw (0,0) -- (1,0) -- (.5,1) -- cycle;
    \end{scope} 
    \begin{scope}[shift={(.5,1)}]
      \draw (0,0) -- (1,0) -- (.5,-1) -- cycle;
    \end{scope} 
    \begin{scope}[shift={(2,0)}]
      \draw (0,0) -- (1,0) -- (.5,1) -- cycle;
    \end{scope} 
    \end{scope}
    
    \begin{scope}[shift={(5,-4)}]
    \fill (1,0) circle (\r);
    \begin{scope}[scale = 2]
      \draw (0,0) -- (1,0) -- (.5,1) -- cycle;
    \end{scope} 
    \begin{scope}[shift={(.5,1)}]
      \draw (0,0) -- (1,0) -- (.5,-1) -- cycle;
    \end{scope}
    \begin{scope}[shift={(-1,0)}]
      \draw (0,0) -- (1,0) -- (.5,1) -- cycle;
    \end{scope}  
    \end{scope}
    
    \begin{scope}[shift={(11,-4)}]
    \fill (0,0) circle (\r);
    \begin{scope}[scale = 2]
      \draw (0,0) -- (1,0) -- (.5,1) -- cycle;
    \end{scope} 
    \begin{scope}[shift={(.5,1)}]
      \draw (0,0) -- (1,0) -- (.5,-1) -- cycle;
    \end{scope}
    \begin{scope}[shift={(-2,0)}]
    \begin{scope}[scale = 2]
      \draw (0,0) -- (1,0) -- (.5,1) -- cycle;
    \end{scope} 
    \begin{scope}[shift={(.5,1)}]
      \draw (0,0) -- (1,0) -- (.5,-1) -- cycle;
    \end{scope}  
    \end{scope} 
    \end{scope}
    
    \end{scope}              
\end{tikzpicture}
\caption{All shapes $\Lambda \in \cS$ for $L_0 = 1$ and $L_0 = 2$. Instead of the sites, the edges (all of unit length) connecting them are shown. The dot indicates the origin in $\R^2$. Each illustrated shape appears 3 times in $\cS$ under the rotations by $0$, $120$, and $240$ degrees.  }
\label{fig:S}
\end{figure}

Next, we introduce the weighted average $\Phi(\rho)$ of the birth-death rate under the assumption of independent site occupancies. It will appear in the hydrodynamic limit. Let $\nu_\rho$ be the product Bernoulli measure on $\{0,1\}^V$ with mean $\rho \in [0,1]$ for a given finite index set $V$. For instance, if $V = V_N$, then $\nu_\rho$ is a measure on $\Omega_N$. Since $\nu_\rho$ is a product measure, the dependence on $V$ is of little importance; we only mention $V$ in those cases where it might not be clear from the context. For a probability measure $\nu$ on $\Omega_N$ and a function $f$ on $\Omega_N$, we denote the expectation of $f$ with respect to $\nu$ by $E_\nu[f]$. In addition, for $\Lambda \in \cS$, let 
\[
  r_\Lambda^N := \frac{|\{ x \in V_N^0 : 2^N(\Lambda_x - x) = \Lambda \}|}{|V_N^0|}
\]
be the ratio of sites $x$ whose neighborhood $\Lambda_x$ corresponds to $\Lambda$. Finally, 
\begin{equation} \label{Phi}
  \Phi(\rho)
  := \sum_{\Lambda \in \cS} r_\Lambda E_{\nu_\rho}[ (1 - 2 \xi_0) c(\xi; \Lambda) ], \qquad 
  r_\Lambda := \lim_{N \to \infty} r_\Lambda^N \in [0,1],
\end{equation}
where $\xi \in \{0,1\}^\Lambda$ (and thus $\nu_\rho$ is a probability measure on $\{0,1\}^\Lambda$). Clearly, $\sum_{\Lambda \in \cS} r_\Lambda^N = 1$ for all $N \geq 1$, and thus $\sum_{\Lambda \in \cS} r_\Lambda = 1$. Hence, $\Phi(\rho)$ is a weighted average of expectations over $\Lambda \in \cS$. Since $\{0,1\}^\Lambda$ is a finite set and since $\rho \mapsto \nu_\rho(\xi)$ is smooth on $[0,1]$ for each $\xi \in \{0,1\}^\Lambda$, we have that $\Phi \in C^\infty([0,1])$.

\begin{ex}
A simple example of $c_x(\eta)$ which satisfies Assumption \ref{a:cx}, is that in \cite[(2.21a)]{DFL}, which is defined on the discrete torus instead of $V_N$. It fits our setting with $L_0 = 1$. Then, $\cS$ only contains 3 rotated copies $\Lambda$ of the same shape; see Figure \ref{fig:S}. We take $c(\xi) := c(\xi; \Lambda)$ independent of $\Lambda$, and let $c(\xi)$ only depend on $\xi_0 := \xi(0)$  and on $\xi_1$ and $\xi_2$ given by the occupancies at the two sites in $\Lambda$ on opposite sides of the origin. We characterize \comm{[p.105 own notes for DFL86 (2.21a)]}
\begin{equation*}
  c(\xi) =
  \tilde c(\xi_0, \xi_1, \xi_2)
  := \left\{ \begin{aligned}
    &1 - 2\gamma + \gamma^2
    &&\text{if } \xi_0 = \xi_1 = \xi_2  \\
    &1 + 2\gamma + \gamma^2
    &&\text{if } \xi_0 \neq \xi_1 = \xi_2  \\
    &1 - \gamma^2
    &&\text{if } \xi_1 \neq \xi_2,
  \end{aligned} \right.
\end{equation*}
where $\gamma \in [0,1)$ is a parameter. The order of $\xi_1, \xi_2$ does not matter since $\tilde c(\xi_0, \xi_1, \xi_2) = \tilde c(\xi_0, \xi_2, \xi_1)$. From the range of $\gamma$ it follows that $c > 0$. Then, defining $c_x(\eta)$ from \eqref{cx} makes $c_x(\eta)$ satisfy Assumption \ref{a:cx} automatically.
A simple computation (see \cite[(2.22)]{DFL}) shows that
\comm{[details: Since the input of $c$ can be interpreted as a binary number between 0 and 7, we will alternatively write its values as $c_{ijk}$ with $i,j,k \in \{0,1\}$. $\bar \rho := 1-\rho$ 
\begin{align*} 
  \Phi(\rho)
   &= E_{\nu_\rho}[ (1 - 2 \xi_x) c_x ( \xi ) ] \\
   &= \bar \rho^3 c_{000} 
      + \bar \rho^2 \rho (c_{001} + c_{010} - c_{100})
      + \bar \rho \rho^2 (c_{011} - c_{101} - c_{110})
      - \rho^3 c_{111} \\ 
   &= \gamma^2 (\bar \rho^3 - 3 \bar \rho^2 \rho + 3 \bar \rho \rho^2 - \rho^3) 
      + 2\gamma (-\bar \rho^3 - \bar \rho^2 \rho + \bar \rho \rho^2 + \rho^3) 
      + (\bar \rho^3 + \bar \rho^2 \rho - \bar \rho \rho^2 - \rho^3) \\
   &= \gamma^2 (\bar \rho - \rho)^3 + (2\gamma - 1)(\bar \rho + \rho)( \rho^2 - \bar \rho^2 ) \\ 
   &= - \gamma^2 (2\rho - 1)^3 + (2\gamma - 1)(2\rho - 1) \\
   &= - \gamma^2 u \Big( u^2 - \frac{2 \gamma - 1}{\gamma^2} \Big) \Big|_{u = 2\rho - 1}
\end{align*}
}
\begin{equation*}
  \Phi(\rho)
  = E_{\nu_\rho}[ (1 - 2 \xi_0) c ( \xi ) ]
  = - \gamma^2 (2\rho-1) \Big( (2\rho-1)^2 - \frac{2 \gamma - 1}{\gamma^2} \Big),
\end{equation*}
which is a third-order polynomial of $\rho$. It is decreasing if $\gamma \leq \frac12$ and has 3 roots if $\gamma > \frac12$.
\end{ex}

A more natural example is $c_x(\eta) = \exp(-\beta \sigma(x) \sum_{y \sim x} \sigma(y))$ with $\beta > 0$ and $\sigma := 2 \eta - 1 \in \{-1,1\}$. It also satisfies Assumption \ref{a:cx} with $L_0 = 1$. It fits to the Ising model, where $\beta$ is the inverse temperature.  Since the number of sites connected to any $x \in V_N^0$ is $4$, $c_x(\eta)$ is the same as for the Ising model on the lattice $\Z^2$. The corresponding function $\Phi(\rho)$ can be computed explicitly; it is a 5th order polynomial of $\rho$. 

\subsection{Discrete and Continuum Laplacians on the Sierpi\'nski gasket}
\label{s:HDL:Eqn:calc}

This section is mainly a rephrased and reduced version of \cite[Section 3.1]{CG}, which is in turn a collection of results from the textbooks \cite{Kig, Str}. We add to it a few basic properties of harmonic functions from the same textbooks. 

Let $N \geq 0$, $f,g : K \to \R$, and recall the notation from Section \ref{s:SG}. The discrete Laplacian $\Delta_N f : V_N^0 \to \R$ and outward normal derivative $\partial_N^\perp f : V_0 \to \R$ on $\cG_N$ are given by
\begin{align*} 
    \Delta_N f (x) := 5^N \sum_{\substack{ y \in V_N \\  y \sim x }} [f(y) - f(x)], \qquad 
    \partial_N^\perp f (a) := \frac{5^N}{3^N} \sum_{\substack{ y \in V_N \\  y \sim a }} [f(a) - f(y)],
\end{align*}
where we recall $V_N^0 = V_N \setminus V_0$, that $5^N$ is the diffusive time scale and that $3^N$ is the scaling of $|V_N|$ and $|E_N|$ (recall \eqref{VN:EN:abs}). The Dirichlet energy is given by
\begin{equation} \label{En:uu}
  \cE_N(f) := \frac12 \frac{5^N}{3^N} \sum_{x  \in V_N} \sum_{\substack{ y \in V_N \\  y \sim x }} [f(y) - f(x)]^2.
\end{equation}
It induces the symmetric quadratic form
\begin{equation*}
  \cE_N(f,g) := \frac14 \big( \cE_N(f+g) - \cE_N(f-g) \big),
\end{equation*}
which is an inner product on the function space $\{ f : V_N \to \R \mid f|_{V_0} = 0 \}$. The connection with $\Delta_N$ and $\partial_N^\perp$ is given by the following integration by parts formula:
\begin{equation} \label{cEN:IBP}
  \cE_N(f,g) = -\frac32 \frac1{|V_N|} \sum_{x  \in V_N^0} \Delta_N f (x) \, g(x) + \sum_{a  \in V_0} \partial_N^\perp f (a) \, g(a).
\end{equation}
Note that all expressions above also apply to functions $f,g$ that are only defined on $V_N$. 

Let $C(K)$ be the space of continuous functions on $K$. We call $f \in C(K)$ $N$-harmonic if $\Delta_M f(x) = 0$ for all $M \geq N+1$ and all $x \in V_M \setminus V_{M-1}$. A $0$-harmonic function is simply called harmonic. 
If $f$ is $N$-harmonic, then $\cE_M(f) = \cE_N(f)$ for all $M \geq N$, and on any $N$-cell $K_w$ (say with boundary points $x \sim y \sim z \sim x$ on $V_N$) the min and max of $f |_{K_w}$ are attained at $x$, $y$ or $z$. Also, for each $f_N : V_N \to \R$ there exists a unique $N$-harmonic function $f$ with $f |_{V_N} = f_N$, which is called the $N$-harmonic extension of $f_N$.

The continuous analogues of $\cE_N, \Delta_N, \partial_N^\perp$ are obtained by passing to the limit $N \to \infty$. The key to passing to this limit is the property that $\cE_N(f)$ is increasing as a function of $N$. Then,
\begin{equation*}
  \cE(f) := \lim_{N \to \infty} \cE_N(f)
\end{equation*}
exists in $[0,\infty]$. We set
\begin{equation*}
  \cF := \{ f : K \to \R \mid \cE(f) < \infty \}
\end{equation*}
to be the domain of $\cE$, and take, in analogy to $\cE_N$,
\begin{equation*}
  \cE(f,g) := \frac14 \big( \cE(f+g) - \cE(f-g) \big)
  \qquad \text{for } f,g \in \cF.
\end{equation*}
Note that for each $N \geq 0$, all $N$-harmonic functions are in $\cF$.
It is known that $\cF \subset C(K)$. In fact, $\cF$ is a Hilbert space. To introduce it, we first recall the Hilbert space $L^2(K)$ induced by the inner product
\begin{equation*}
  \lrang{f,g}_{L^2(K)} := \int_K fg \, dm.
\end{equation*}
Then, 
\begin{equation} \label{cE1}
  \cE_1(f,g) := \cE(f,g) + \lrang{f,g}_{L^2(K)}
\end{equation}
defines an inner product on $\cF$. Finally, for functions $F,G : [0,T] \times K \to \R$, we let $L^2(\cF) := L^2(0,T; \cF)$ be the Hilbert space generated by the inner product
\begin{equation*}
  \lrang{F,G}_{L^2(\cF)} := \int_0^T \cE_1(F_t, G_t) \, dt.
\end{equation*}
This will be the function space on which we define solutions to the reaction--diffusion equation in the hydrodynamic limit.

Next, we define the Laplacian $\Delta$ on $K$. It is motivated by \eqref{cEN:IBP} applied with $g|_{V_0} = 0$. The domain of $\Delta$, denoted by $\cD(\Delta)$, is the space of functions $u \in \cF$ for which there exists $f \in C(K)$ such that
\begin{equation*}
  \cE(u,\varphi) = \int_K f \varphi \, dm
  \qquad \text{for all } \varphi \in \cF_0 := \{g \in \cF : g|_{V_0} = 0 \}. 
\end{equation*}
For each $u \in \cD(\Delta)$, the corresponding $f \in C(K)$ is unique; we denote it by $-\Delta u$. This defines the Laplacian $\Delta$ on $K$. We further set 
\begin{equation*}
  \cD_0(\Delta) := \{ u \in \cD(\Delta) \mid u|_{V_0} = 0 \}, \quad
  \Delta_0 := \Delta |_{\cD_0(\Delta)}.
\end{equation*}

Finally, we list three properties related to $\Delta$. For any $u \in \cD(\Delta)$,
\begin{align} \label{DelN:to:Del}
  \frac32 \Delta_N u &\to \Delta u && \text{uniformly on } K \setminus V_0, \\\label{delNp:to:delp}
  \partial^\perp u(a) &:= \lim_{N \to \infty} \partial_N^\perp u(a) && \text{exists for each } a \in V_0, \\ \notag
  \cE(u,g) &= - \int_K \Delta u \, g \, dm + \sum_{a  \in V_0} \partial^\perp u(a) \, g(a) && \text{for all } g \in \cF. 
\end{align}

\subsection{The reaction--diffusion equation with Dirichlet boundary condition}
\label{s:HDL:Eqn:HDL}

The reaction--diffusion equation with a Dirichlet boundary condition is given by
\begin{equation} \label{HDL}
  \left\{ \begin{aligned}
    \partial_t \rho(t,x) &= \frac23 \Delta \rho(t,x) + \Phi(\rho(t,x))
    &&t \in (0,T], \ x \in K \setminus V_0  \\
    \rho(t, a) &= \rho_B(a)
    &&t \in [0,T], \ a \in V_0 \\
    \rho(0, x) &= \rho_\circ(x)
    &&x \in K,
  \end{aligned} \right.
\end{equation} 
where $T > 0$ is the given end time, $\rho_B : V_0 \to \R$ is the given boundary datum, $\rho_\circ \in L^2(K)$  is the given initial condition, and $\Phi : \R \to \R$ is a given Lipschitz continuous function that describes the (nonlinear) reaction term. Under these weak assumptions on the data, the unknown function $\rho$ may attain any value in $\R$. We will often write $\rho_t(x) = \rho(t,x)$. 
We will consider the same equation with Robin or Neumann boundary conditions in Section \ref{s:b}.

We will only focus on weak solutions of \eqref{HDL}, and refer to them simply as solutions in most of the sequel.

\begin{defn}[Weak solution] \label{d:wSol}
A measurable function $\rho:[0,T]\times K\to\R$ is a weak solution of \eqref{HDL} if
\begin{enumerate}
  \item $\rho \in L^2(\cF)$,
  \item $\rho_t|_{V_0} = \rho_B$ for a.e.\ $t \in (0,T)$, and
  \item For all $t \in (0,T)$ and all test functions $F \in \cD_T := C^1((0,T); \cD_0(\Delta)) \cap C([0,T]; \cD_0(\Delta))$ 
  \begin{multline} \label{weak-form}
    0
    = \Theta_t(\rho, F) 
    := \int_K \rho_t F_t \, dm 
      - \int_K \rho_\circ F_0 \, dm
      - \int_0^t \int_K \rho_s \Big( \frac23 \Delta + \partial_s \Big) F_s \, dm ds 
      \\
      - \int_0^t \int_K \Phi(\rho_s) F_s \, dm ds
      + \frac23 \int_0^t \sum_{a \in V_0} \rho_B(a) \partial^\perp F_s(a)  \, ds.
  \end{multline}
\end{enumerate}
\end{defn}

The existence of solutions will be a by-product of the hydrodynamic limit; see Theorem \ref{t}. We prove it for a restricted choice of the data $\rho_B, \rho_\circ, \Phi$. The following uniqueness result holds for the general class of data as introduced above. Its proof will be given in Section \ref{s:uniqueness}.

\begin{thm} \label{t:HDL:un}
Weak solutions of \eqref{HDL} are unique for any $T, \rho_B, \rho_\circ, \Phi$ in the setting above. 
\end{thm}

\subsection{The hydrodynamic limit for $b < \frac53$}

The value $b = \frac53$ in the particle system is critical for how the dynamics at the boundary $V_0$ manifest themselves in the hydrodynamic limit as the boundary condition for the reaction--diffusion equation. This is elucidated in \cite{CG} in the absence of the Glauber part. Roughly speaking, $b < \frac53$, $b = \frac53$, and $b > \frac53$ result in a Dirichlet, Robin and Neumann condition, respectively. We focus on the case $b < \frac53$ in the bulk of this paper, and comment on the cases $b = \frac53$ and $b > \frac53$ afterwards in Section \ref{s:b}.

To state our main result (Theorem \ref{t}) for $b < \frac53$, we introduce some notation on the Glauber-Kawasaki dynamics $\eta_\bullet^N := (t \mapsto \eta_t^N)$ for $t \in [0,T]$ with $T > 0$ fixed as defined in Section \ref{s:GK}. We consider large $N \in \N$ as a variable. Let $\mu_N$ be the initial distribution of the process (on $\Omega_N$) and $\P_{\mu_N}$ be the distribution on the Skorokhod space $D([0, T ], \Omega_N )$ of the process.  
Let $\E_{\mu_N}$ be the expectation with respect to $\P_{\mu_N}$. 
As usual, we say that $(\mu_N)_{N \geq 1}$ is
associated with $\rho \in L^1(K)$  if for any $f \in C(K)$ and any $\e > 0$
\begin{equation} \label{asso-d:meas-s}
  \lim_{N \to \infty} \mu_N \Big( \eta \in \Omega_N : \Big| \frac1{|V_N|} \sum_{x \in V_N} f(x) \eta(x) - \int_K f \rho \, dm \Big| > \e \Big) = 0,
\end{equation}
where we recall that $m$ is the self-similar probability measure on $K$.

The following theorem is the main result of this paper.

\begin{thm}[Hydrodynamic limit] \label{t}
Let $b < \frac53$, $\lambda_\pm$ and $c_x$ be as in Section \ref{s:GK}. Let $\rho_B := \frac{\lambda_+}{\lambda_+ + \lambda_-} : V_0 \to (0,1)$ and $\Phi : [0,1] \to \R$ be defined from $c_x$ as in \eqref{Phi}. Let $T > 0$ and $\rho_\circ $ be as in Section \ref{t:HDL:un} such that $0 \le \rho_\circ \le 1$. Then, there exists a unique weak solution $\rho_t$ of \eqref{HDL} with data $T, \rho_\circ, \rho_B, \Phi$ such that $\rho_t(x) \in [0,1]$ for all $x \in K$ and a.e.\ $t \in (0,T)$.
Moreover, for any $(\mu_N)_{N \geq 1}$ associated to $\rho_\circ$ (see above), any $t \in [0,T]$, any $\e > 0$ and any $f \in C(K)$
\begin{equation*}
  \lim_{N \to \infty} \P_{\mu_N} \Big( \Big| \frac1{|V_N|} \sum_{x \in V_N} f(x) \eta_t^N(x) - \int_K f \rho_t \, dm \Big| > \e \Big) = 0.
\end{equation*}
\end{thm} 

\section{Proof of Theorem \ref{t}}\label{s:pf}

Apart from a key lemma, this section contains the proof of Theorem \ref{t}. This section is a minor modification of the proof of the hydrodynamic limit in \cite{CG}, which we recall to be the restriction of Theorem \ref{t} to the case without the Glauber part, i.e.\ $c_x = 0$ and $\Phi = 0$. Here, we recall several parts of that proof to highlight where modifications are needed. The key lemma that we leave out of this section does not appear in \cite{CG}. It is a replacement lemma for the Glauber part, which we treat in detail in Section \ref{s:repl}. 

\subsection{Preparation: lower bounds on Dirichlet forms}
\label{s:pf:Dform}

Let
\begin{equation*}
  \rho_* := \min_{V_0} \rho_B, \qquad \rho^* := \max_{V_0} \rho_B.
\end{equation*}
Since $\rho_B = \frac{\lambda_+}{\lambda_+ + \lambda_-}$ and $\lambda_\pm(a) > 0$, we have $0 < \rho_* \leq \rho^* < 1$. Let $\varrho : K \to \R$ be  such that
\begin{equation} \label{pfum}
  \varrho \in \cF,
  \quad \varrho |_{V_0} = \rho_B,
  \quad \min_K \varrho = \rho_*,
  \quad \max_K \varrho = \rho^*.
\end{equation}
For a finite set $V$ and a probability measure $\nu$ on $\{0,1\}^V$,
we denote the $L^2(\nu)$-inner product by $\langle\cdot,\cdot\rangle_\nu$.
Moreover, we say that for a nonnegative function $f:\{0,1\}^V\to[0,\infty)$, $f$ is a density with respect to $\nu$ if $E_\nu[f]=1$.
Corresponding to $\varrho$, we introduce the `Dirichlet form'
\begin{equation*}
  \cD_N(\psi^2, \varrho) := \lrang{ \psi, - \cL_N  \psi }_{\nu_\varrho}
\end{equation*}
for densities $\psi^2$ with respect to $\nu_\varrho$. It is not an actual Dirichlet form, because it is neither symmetric nor nonnegative. Yet, it is convex with respect to $\psi^2$. 
We also introduce the carr\'e du champ of the Kawasaki part given by
\begin{equation*}
  \Gamma_N (\psi^2, \varrho) := \frac12 E_{\nu_\varrho} \Big[ \sum_{x \in V_N} \sum_{ \substack{ y \in V_N \\ y \sim x } } [\psi(\eta^{xy}) - \psi(\eta)]^2 \Big]
\end{equation*}
for densities $\psi^2$ with respect to $\nu_\varrho$. It is nonnegative and convex with respect to $\psi^2$. 

In \cite[Corollary 5.4]{CG} the following lower bounds for any density $\psi^2$ with respect to $\nu_\varrho$ are established (recall $\cE_N$ from \eqref{En:uu}): 
  \begin{align} \label{pfup}
  \lrang{ \psi, - \cL_N^K  \psi }_{\nu_\varrho} 
  &\geq\Gamma_N (\psi^2, \varrho) - C \Big( \frac35 \Big)^N \cE_N(\varrho), \\\label{pfuo}
  \lrang{ \psi, - b^{-N} \cL_N^B  \psi }_{\nu_\varrho} 
  &= \frac1{2 b^{N}} \int_{\Omega_N} \sum_{a \in V_0} \big( \lambda_-(a)\eta(a) + \lambda_+(a) (1 - \eta(a)) \big) [\psi(\eta^a) - \psi(\eta)]^2 \, d\nu_\varrho,
\end{align}
where $C = C(\rho_B) > 0$ is a constant.
Note that $\lrang{ \psi, - b^{-N} \cL_N^B  \psi }_{\nu_\varrho}$ is clearly nonegative. We need to add to this a new bound on the Dirichlet form of $\cL_N^G$:   

\begin{prop}
\label{p:DNG:LB}
Let $N \geq 1$ and $\varrho$ be as above. Then, there exists a constant $C > 0$ which only depends on $\rho_B$ and $\| c \|_\infty$ (recall \eqref{cinf}) such that for any density $\psi^2$ with respect to $\nu_\varrho$
\begin{equation*}
  \lrang{ \psi, - 5^{-N} \cL_N^G \psi }_{\nu_\varrho}
  \geq -C \Big( \frac35 \Big)^N.
\end{equation*}
\end{prop}

\begin{proof}
Writing $
  \psi^x(\eta) := \psi(\eta^x)$ and $
  \nu_\varrho^x(\eta) := \nu_\varrho(\eta^x)$, a simple computation yields
\begin{gather*} 
  \begin{aligned}
    \lrang{ \psi, - \cL_N^G \psi }_{\nu_\varrho}
    &= \int_{\Omega_N} \sum_{x \in V_N^0} c_x \psi (\psi - \psi^x) \,d \nu_\varrho\\
   &\geq \frac12 \int_{\Omega_N} \sum_{x \in V_N^0} c_x (\psi^2 - (\psi^x)^2) \,d \nu_\varrho \\
   &\geq -\frac12 \Big( \max_{\eta \in \Omega_N } \max_{y \in V_N^0} c_y(\eta) \Big) \int_{\Omega_N} \sum_{x \in V_N^0}  (\psi^x)^2 \,d \nu_\varrho \\
   &= -C \sum_{x \in V_N^0} \int_{\Omega_N} \psi^2 \frac{\nu_\varrho^x}{\nu_\varrho} \,d \nu_\varrho.
  \end{aligned}  
\end{gather*}
The result then follows from applying, for any $x \in V_N^0$
\begin{equation*}
  \frac{\nu_\varrho^x}{\nu_\varrho}(\eta)
  = \left\{ \begin{aligned}
    &\frac{1 - \varrho(x)}{\varrho(x)} 
    && \text{if } \eta_x = 1 \\
    &\frac{\varrho(x)}{1 - \varrho(x)} 
    && \text{if } \eta_x = 0
  \end{aligned} \right\}
  \leq \frac1{\min \varrho} \vee \frac1{\min (1 - \varrho)} 
  = \frac1{\rho_*} \vee \frac1{1-\rho^*},
\end{equation*}
which completes the proof.
\end{proof}

\subsection{The choice of the reference density $\varrho$}
\label{s:pf:varrCG}

In the remainder of Section \ref{s:pf}, we take the reference density $\varrho$ to be the harmonic function characterized by $\varrho |_{V_0} = \rho_B$. Recalling Section \ref{s:HDL:Eqn:calc} for the definition and properties of harmonic functions on $K$, we observe that $\varrho$ satisfies the requirements in \eqref{pfum}. 
In addition, $\cE_N(\varrho) = \cE_0(\varrho) = \cE_0(\rho_B)$ for all $N \geq 0$. Using this with the bounds obtained in Section \ref{s:pf:Dform}, we get
\begin{multline} \label{pfun}
  \cD_N(\psi^2, \varrho)
  \geq \Gamma_N (\psi^2, \varrho) - C \Big( \frac35 \Big)^N \\
  + \frac1{2 b^{N}} \int_{\Omega_N} \sum_{a \in V_0} \big( \lambda_-(a)\eta(a) + \lambda_+(a) (1 - \eta(a)) \big) [\psi(\eta^a) - \psi(\eta)]^2 \, d\nu_\varrho
\end{multline}
for any density $\psi^2$ with respect to $\nu_\varrho$.
We note that \eqref{pfun} is, qualitatively in $N$, the same lower bound as in \cite{CG}, i.e.\ the case without the Glauber term.

\subsection{The associated martingale}
\label{s:pf:mg}

This section follows the lines of \cite[Section 3.4]{CG}. The only new parts are the computations involving the Glauber term.

We introduce further notation. To the Markov process $\eta_\bullet^N$ we associate the  random empirical measure on $K$ given by 
\begin{equation*}
  \pi_t^N := \frac1{|V_N|} \sum_{x \in V_N} \eta_t^N (x) \delta_x.
\end{equation*}
In particular, it acts on $f \in C (K)$ as
\begin{equation*}
  \pi_t^N (f) := \frac1{|V_N|} \sum_{x \in V_N} \eta_t^N (x) f(x).
\end{equation*}
Let $\cM_+(K)$ be the space of nonnegative measures on $K$, endowed with the weak topology, with total mass bounded by $1$. Note that $\pi_t^N$ is a $\cM_+(K)$-valued random variable. We also consider the (random) path $\pi_\bullet^N := \{\pi_t^N : t \in [0,T]\}$ in the Skorokhod space $D([0,T]; \cM_+(K))$, and denote by $\Q_N$ the law of $\pi_\bullet^N$ on $D([0,T]; \cM_+(K))$. An alternative description for the display in Theorem \ref{t} is 
\begin{equation} \label{pful}
  \lim_{N \to \infty} \Q_N \Big( \pi_\bullet^N \in D([0,T]; \cM_+(K)) : \Big| \pi_t^N(f) - \int_K f \rho_t \, dm \Big| > \e \Big) = 0.
\end{equation}

Let $N \geq 1$ and $F \in \cD_T$ (recall Definition \ref{d:wSol}). Then, by Dynkin's formula, the following process constructed from $\eta_t^N$ is a martingale on $[0,T]$: \comm{[Dynkin's formula: KL Lem A1.5.1. KL writes $F_t : E \to \R$ instead of our $\pi_t^N(F_t) : \Omega_N \to \R$. KL assumes $F$ to be $C^1$ in $[0,T]$, but the proof allows for $C^1$ only on the interior, and continuity up to $0$ and $T$.] } 
\begin{equation} \label{MtNF}
  M_t^N(F) 
  := \pi_t^N(F_t) - \pi_0^N(F_0) - \int_0^t (\partial_s + 5^N \cL_N) \pi_s^N(F_s) \, ds,
\end{equation}
where in the integrand $\pi_s^N(F_s)$ is to be understood as $\varphi(s, \eta_s^N)$ with $\varphi : [0,T] \times \Omega_N \to \R$. Thus, $\partial_s \pi_s^N(F_s) = \pi_s^N(\partial_s F_s)$ and (by linearity of $\cL_N$)
\begin{align*}
  5^N \cL_N \pi_s^N(F_s)
  = \frac1{|V_N|} \sum_{x \in V_N} F_s(x) 5^N \cL_N (\eta_s^N (x)).
\end{align*}
A long, but elementary, computation reveals that 
\begin{align} \notag
  5^N \cL_N^K \pi_s^N(F_s)
  &= \frac1{|V_N|} \sum_{x \in V_N^0} \eta_s^N (x) \Delta_N F_s(x)
    - \frac{3^N}{|V_N|} \sum_{a \in V_0} \eta_s^N (a) \partial_N^\perp F_s(a), \\  
     \label{pfuq} 
     \cL_N^G \pi_s^N(F_s)
  &= \frac1{|V_N|} \sum_{x \in V_N^0}  F_s(x) c_x(\eta_s^N) (1 - 2 \eta_s^N (x)),
\end{align} 
and $\cL_N^B \pi_s^N(F_s)=0$.
Putting all three together, we obtain
\begin{align} \label{pfzz}
  5^N \cL_N \pi_s^N(F_s)
  &= \frac1{|V_N|} \sum_{x \in V_N^0} \Big( \eta_s^N (x) \Delta_N F_s(x) + c_x(\eta_s^N) (1 - 2 \eta_s^N (x)) F_s(x) \Big) \notag\\
  &\qquad - \frac{3^N}{|V_N|} \sum_{a \in V_0} \eta_s^N (a) \partial_N^\perp F_s(a).
\end{align}
 
For later use, we also consider the quadratic variation \comm{[CG21 computes this only for a t-indep test fct in (3.23). But, the same formula holds for t-dep test fcts. See (3.25) or otherwise KL App 1.5.]}
\begin{equation*}
  \lrang{M^N(F)}_t
  = \int_0^t 5^N [ \cL_N \big( \pi_s^N(F_s)^2 \big) - 2 \pi_s^N(F_s) \cL_N \pi_s^N(F_s) ] \, ds. 
\end{equation*}
Since this is linear in $\cL_N$, we refer to the computation of the Kawasaki and boundary part to \cite{CG}, and compute for the Glauber part (with short-hand notation $\sum_x = \sum_{x \in V_N^0}$, $F_x = F_s(x)$, $\eta_x = \eta_s^N(x)$, $\hat \eta_x := 1 - 2\eta_x$, $c_x = c_x(\eta)$)
\begin{align*}
  2 \cdot 5^N \pi_s^N(F_s) \cL_N^G \pi_s^N(F_s)
  &= \frac2{|V_N|^2} \Big( \sum_x \eta_x F_x \Big) \Big( \sum_y F_y c_y \hat \eta_y \Big) \\
  &= \frac2{|V_N|^2} \sum_{x,y} F_x F_y c_y \eta_x \hat \eta_y
\end{align*}
and
\begin{align*}
  5^N \cL_N^G \big( \pi_s^N(F_s)^2 \big)
  &= \frac1{|V_N|^2} \sum_{x,y} F_x F_y \cL_N^G (\eta_x \eta_y) \\
  &= \frac1{|V_N|^2} \Big( \sum_{x,z} \sum_{y \neq x} F_x F_y \big( \delta_{xz} c_z \eta_y \hat \eta_x + \delta_{yz} c_z \eta_x \hat \eta_y \big) + \sum_{x,z} F_x^2 \delta_{xz} c_z \hat \eta_x \big) \Big) \\
  &= \frac2{|V_N|^2} \sum_{x} \sum_{y \neq x} F_x F_y c_y \eta_x \hat \eta_y 
     + \frac2{|V_N|^2} \sum_x F_x^2 c_x \hat \eta_x \\
  &= \frac2{|V_N|^2} \sum_{x,y} F_x F_y c_y \eta_x \hat \eta_y 
     + \frac2{|V_N|^2} \sum_x F_x^2 c_x (1-\eta_x),   
\end{align*} 
where we recognize the first term from the display above. We obtain
\begin{align} \label{pfuj}
  \lrang{M^N(F)}_t
  &= \int_0^t \frac{5^N}{|V_N|^2} \sum_{x \in V_N} \sum_{ \substack{x \in V_N \\ y \sim x} } \big[ \eta_s^N (x) - \eta_s^N (y) \big]^2 [ F_s(x) - F_s(y) ]^2 ds \notag\\
   &\qquad + \int_0^t \frac2{|V_N|^2} \sum_{x \in V_N^0} c_x(\eta_s^N) (1 - \eta_s^N(x)) F_s(x)^2 \, ds. 
\end{align}
With the bound \cite[(6.5)]{CG} for the Kawasaki part, we get (recall \eqref{cinf} for $c_x$ and \eqref{cE1} for the norm on $\cF$)
\begin{equation} \label{pfui}
  \lrang{M^N(F)}_t
  \leq \frac{ C t}{3^N} \max_{0 \leq s \leq t} \cE_N(F_s) + \frac{C' t}{3^N} \| c \|_\infty  \| F \|_\infty^2
  \leq \frac{C'' t}{3^N}.
\end{equation}

\subsection{Heuristically passing to $N \to \infty$}

This short section follows the lines of \cite[Section 3.4]{CG}. It provides a guideline for the proof of Theorem \ref{t} in the remaining sections, but is not a part of the proof itself.

We formally pass to the limit $N \to \infty$ for the expectation of each of the terms in \eqref{MtNF}. We assume that $\pi_t^N$ converges to a deterministic function $\rho_t : K \to (0,1)$. We expect the martingale term $M_t^N(F)$ to vanish, and terms corresponding to $\pi_s^N(G)$ to converge to $\int_K G_s \rho_s \, dm$. That leaves the $\cL_N$ term, which is expanded in \eqref{pfzz}. In \eqref{pfzz}, the $\Delta_N$-term reads as $\pi_s^N(\Delta_N F_s) + O(1/|V_N|)$, where the error term contains the contribution from the three boundary sites. Since $F_s \in \cD_0(\Lambda)$, we expect from \eqref{DelN:to:Del} that it converges to $\int_K \frac23 \Delta F_s \rho_s \, dm$. From \eqref{Phi} we expect the $c_x$-term to converge to $\int_K F_s \Phi(\rho_s) \, dm$. In the justification of this part, we need our new replacement lemma. Regarding the term involving the sum over $a \in V_0$, we recall from \eqref{VN:EN:abs} that $3^N / |V_N| \to \frac23$ and from \eqref{delNp:to:delp} that $\partial_N^\perp F_s(a) \to \partial^\perp F_s(a)$. Putting together all these expected limits, we observe that $\rho_t$ satisfies the weak form \eqref{weak-form} of the reaction--diffusion equation.

\subsection{Tightness of $\Q_N$} 

The proof for the tightness of $\Q_N$ given in \cite[Section 6.1]{CG} applies directly, given that we have already established the bound \eqref{pfui} on the quadratic variation of the martingale $M_t^N(F)$ and that the summand in \eqref{pfuq} is uniformly bounded by $\|c\|_\infty \|F\|_\infty$.
\comm{[
Here is why: To prove tightness of $\Q_N$, it is, by Aldous' criterion and KL Prop 4.1.7, sufficient to show that for every $f$ in a dense subset of $C(K)$: \comm{[Aldous' criterion: CG21 L6.1]}
\begin{itemize}
  \item[(A1)] for all $t \in [0,T]$
  \begin{equation*}
    \lim_{M \to \infty} \sup_N \Q_N( |\pi_t^N(f)| > M ) = 0,
  \end{equation*}
  \item[(A2)] for all $\e > 0$
  \begin{equation*}
    \lim_{\gamma \to 0} \limsup_{N \to \infty} \sup_{ \tau \in \cT } \sup_{\theta \leq \gamma}
      \Q_N( |\pi_{(\tau + \theta) \wedge T }^N(f) - \pi_\tau^N(f)| > \e ) = 0,
  \end{equation*}
  where $\cT $ is the family of stopping times bounded by $T$ and $d$ is a metric that metrizes the weak topology on $\cM_+(K)$. 
\end{itemize}
Statement (A1) follows directly from $|\pi_t^N(f)| \leq \|f\|_\infty < \infty$. The proof of (A2) applies \eqref{MtNF} to the inside of the absolute value signs, uses the triangle inequality to put absolute value signs around each of the integrals appearing from \eqref{MtNF}, splits $\Q_N(\cdots)$ as a sum over such terms separately for each integral term, and then applies Chebyshev's inequality to bound each such $\Q_N(\cdots)$ by $\E_{\mu_N}[|\cdots|^2]$. To obtain this bound for the Glauber term, it is sufficient to show that the summand in \eqref{pfuq} is uniformly bounded in $N,x,s$. To obtain this bound for the martingale part, \eqref{pfui} is sufficient.
]} 

\subsection{A priori properties of the limit path measure $\Q$}

From the tightness of $\Q_N$ we obtain the existence of a subsequence (not relabeled) and a limit point $\Q$ such that $\Q_N \weakto \Q$. By a standard argument that only relies on $|\eta_s^N(x)|$ being bounded uniformly in $N,s,x$, it follows that $\Q$ is concentrated on paths $\pi_\bullet$ for which $\pi_t$ has for a.e.\ $t$ a density $\rho_t$ with respect to $m$, and that $\rho_0 = \rho_\circ$. Moreover, it follows from  $\eta_s^N(x) \in \{0,1\}$ for all $N,s,x$ and from the properties of $m$ that $\rho_t(x) \in [0,1]$ for a.e.\ $t \in (0,T)$ and a.e.\ $x \in K$.

In \cite[Proposition 6.3]{CG}  it is established that $\rho \in L^2(0,T; \cF)$ $\Q$-a.s.\ when there is no Glauber part. The corresponding proof only requires $\cL_N$ with any added terms, so that the lower bound on the Dirichlet form given by \eqref{pfun} holds. We have already proved this for our setting with the Glauber term. Hence, we have $\rho \in L^2(0,T; \cF)$ $\Q$-a.s.
\comm{[In detail: to get $\rho \in L^2(0,T; \cF)$ $\Q$-a.s., the proof only needs the Laplace term of the weak equation, which comes from the limit of the $\Delta_N$ term in $M_t^N$. Hence, our new Glauber term only appears in the proof in $\E_\Q$, where the path measure $\Q$ is determined from the full process. This surfaces in the proof in L6.6 as the Dirichlet form of $\cL_N$. Now, CG only needs a lower bd on the Dirichlet form.]}

\subsection{The limiting density $\rho$ is a weak solution}
\label{s:pf:wSol}

The key statement that we will prove is that for any $\e > 0$ and any test function $F \in \cD_T$
\begin{equation} \label{pfzw}
   \Q \Big( \sup_{0 \leq t \leq T} |\Theta_t(\rho, F)| > \e \Big) = 0,
\end{equation}
where $\Theta_t(\rho, F)$ is the weak form of the reaction--diffusion equation; see \eqref{weak-form}. In addition, we use Theorem \ref{t:HDL:un} (for the proof, see Section \ref{s:uniqueness}) and show
\begin{equation} \label{pftx}
  \Q(\rho_t|_{V_0} = \rho_B \text{ for a.e.\ } t \in (0,T)) = 1.
\end{equation}
Then, by the uniqueness of weak solutions (see Theorem \ref{t:HDL:un}, which applies since $\Phi$ is smooth on $[0,1]$ (recall the statement below \eqref{Phi}) and $\rho_t(x) \in [0,1]$) and $\Q$ being concentrated on $L^2(0,T; \cF)$, it follows that $\Q = \delta_\rho$, where $\rho$ is the unique weak solution of \eqref{t:HDL:un}. Consequently, the full sequence $\Q_N$ converges to $\Q$, which implies \eqref{pful}. This then completes the proof of Theorem \ref{t}. 

For the proof of \eqref{pftx} we refer to \cite[Lemma 6.8]{CG}. For the proof of this lemma to apply to our setting with the Glauber term, the bound \eqref{pfun} is sufficient.

In the remainder, we focus on proving \eqref{pfzw}, which is done in \cite[Proposition 6.7]{CG} without the Glauber term. We briefly recall the proof to demonstrate that most steps apply to our setting, and to build the setting in which we are going to apply our new replacement lemma for the Glauber term.
\comm{[Their proof does $b > \frac53$, and sketches our case afterwards. 
They reason backwards; from the desired bound on $\Theta_t(\rho, F)$ to the already established bound on $M_t^N(F)$ (KL does it oppositely).]}

The starting point is to show that
\begin{equation} \label{pfur}
   \limsup_{N \to \infty} \Q_N \Big( \sup_{0 \leq t \leq T} |M_t^N(F)| > \e \Big) = 0
\end{equation}
for any $\e > 0$ and any $F \in \cD_T$. By Doob's inequality, the probability in the left-hand side is bounded from above by
$
  \frac4{\e^2} \E_{\mu_N} [\lrang{M^N(F)}_T]
$.
By \eqref{pfui}, this vanishes in the limit $N \to \infty$.  

With \eqref{pfur} established, we work towards \eqref{pfzw}. This would follow from $\Q_N \weakto \Q$ and the Portmanteau Theorem if the arguments of $\sup_t |\cdot|$ in the expressions of $\Q_N$ and $\Q$ were to be equal, and equal to an $N$-independent, bounded function $\Psi$ of $\pi_\bullet^N$ that is continuous in the Skorokhod topology. However, these terms are not equal, and neither $M_t^N(F)$ nor $\Theta_t(\rho, F)$ is continuous. The strategy is to approximate both terms by a single continuous term and show that the error made by this approximation vanishes in the limit. We call such approximation steps `replacements'. By virtue of the triangle inequality, each component of $M_t^N(F)$ and $\Theta_t(\rho, F)$ can be replaced independently.
\comm{[For PvM to recall: \\
Wanted: $\limsup_{N \to \infty} \Q_N( \Phi(\pi_\bullet) > \e ) \geq \Q( \Phi(\pi_\bullet) > \e )$. \\ 
Needs: $\Phi : D \to [0,\infty)$ cts and $\Q_N \weakto \Q$. \\
Pf: $\{ \pi_\bullet : \Phi(\pi_\bullet) > \e\}$ is open; conclude by Portmanteau. \\
\\
Wanted: $\Phi(\pi_\bullet) = \sup_{0 \leq t \leq T} | \pi_0(f) + \pi_t(f) + \int_0^t F(\pi_s) \, ds|$ cts. \\
Needs: $f \in C(K)$ and $F : \cP(K) \to \R$ bdd and cts in weak topology.
\begin{enumerate}
  \item Tool: $\pi_\bullet^N \to \pi_\bullet$ in Skor top implies $\pi_t^N \weakto \pi_t$ in $\cP(K)$ for a.e.\ $t$ and for $t \in \{0,T\}$ (standard prop of Skor space; KL notes, Chap 4, prop (V))
  \item Tool: KL p56 bottom says OK if $\int_0^t F(\pi_s^N) \, ds \to \int_0^t F(\pi_s) \, ds$ for all $t$
  \item $F$ cts means $F(\pi_s^N) \to F(\pi_s)$ a.e.\ $s$
  \item This with $F$ bdd gives with DCT that $\int_0^t F(\pi_s^N) \, ds \to \int_0^t F(\pi_s) \, ds$ for all $t$
\end{enumerate}
]}
 
To introduce our replacement terms, we introduce some notation. Given $1 \leq M \leq N$, let $w$ be a word of length $M$. Recall from Section \ref{s:SG} the corresponding concatenation of similitudes $\varphi_w$ and the $M$-cells $K_w := \varphi_w(K)$. 
We define 
\begin{equation} \label{pfuh}
  V_N^w := V_N \bigcap K_w \setminus V_M = \varphi_w(V_{N-M}^0)
\end{equation}
and the \textit{approximate} average density on it as
\begin{equation} \label{pfua}
  \pi_s^N(\iota_w) = \frac{3^M}{|V_N|} \sum_{x \in V_N^w} \eta_s^N (x),
  \qquad
  \iota_w(z) := \frac1{m(K_w)} \bone_{K_w}(z) = 3^M \bone_{K_w}(z)
\end{equation}
with $z \in K$. The actual average density on $V_N^w$ is
\begin{equation} \label{pfvm}
  \ov \eta_s^N (V_N^w) := \frac1{|V_N^w|} \sum_{x \in V_N^w} \eta_s^N (x).
\end{equation}
It only differs from $\pi_s^N(\iota_w)$ through the different prefactors
\begin{align*}
  \frac1{|V_N^w|} = \frac1{|V_{N- M }^0|} = \frac2{3(3^{N-M} - 1)}, \qquad
  \frac{3^M}{|V_N|} = \frac{2}{3 ( 3^{N-M} - 3^{-M})},
\end{align*}
whose difference vanishes in the limit $N \to \infty$.

Note that $\pi_s^N(\iota_w)$ is not continuous in the Skorokhod topology, because $\iota_w$ is not continuous. We handle this by approximating $\iota_w$ from above by continuous functions $\iota_w^k$ as $k \to \infty$. A natural choice for $\iota_w$ is the $(M+k)$-harmonic extension (recall Section \ref{s:HDL:Eqn:calc}) which equals $1$ on $V_{M+k} \cap K_w$ and $0$ on all other sites of $V_{M+k}$. Then, $\iota_w^k$ is $1$ on $K_w$, and $0$ on $K \setminus K_w$ except for the three (or two, if $K_w \cap V_0 \neq \emptyset$) smaller $(M+k)$-cells adjacent to $K_w$ (this requires $k \leq N-M$), where $\iota_w^k$ continuously connects the `boundary values' $0$ and $1$. It is easy to see that 
\begin{subequations} \label{pfuc}
\begin{equation} \label{pfuca}
  \| \iota_w^k - \iota_w \|_{L^1(K,m)} \to 0
\end{equation}
as $k \to \infty$ with $M$ fixed, and that
\begin{equation} \label{pfucb}
  |\pi_s^N(\iota_w^k) - \pi_s^N(\iota_w)| \leq \| \iota_w^k - \iota_w \|_{L^1(K,m)}
\end{equation}
\end{subequations}
for all $N$ large enough (with respect to $M$ and $k$) and all $s \in [0,T]$.

With $\iota_w$ and $\iota_w^k$ introduced, we turn to the replacement statements that we need. First, $M_t^N(F)$ in \eqref{pfur} reads as (recall \eqref{MtNF} and \eqref{pfzz})
\begin{multline} \label{pfug} 
  M_t^N(F)
  = \pi_t^N(F_t) 
- \pi_0^N(F_0) 
- \int_0^t \pi_s^N \big( (\partial_s + \Delta_N^0) F_s \big) \, ds \\
- \int_0^t \frac1{|V_N|} \sum_{x \in V_N^0} c_x(\eta_s^N) (1 - 2 \eta_s^N (x)) F_s(x) \, ds 
+ \int_0^t \frac{3^N}{|V_N|} \sum_{a \in V_0} \eta_s^N (a) \partial_N^\perp F_s(a) \, ds,
\end{multline}
where 
\begin{equation*}
  \Delta_N^0 f (x) := \left\{ \begin{aligned}
    &\Delta_N f (x) 
    &&\text{if } x \in V_N^0 \\
    &0
    &&\text{if } x \in V_0
  \end{aligned} \right.
\end{equation*}
for any $f : V_N \to \R$.

To demonstrate how the replacement of several terms in \eqref{pfug} works, we focus in detail on the boundary term, and briefly mention the other, similar replacements afterwards. For the boundary term, we replace $3^N |V_N|^{-1}$ by $\frac23$, $\partial_N^\perp F_s(a)$ by $\partial^\perp F_s(a)$ and $\eta_s^N (a)$ by $\pi_s^N(\iota_{w^a}^k)$, where $w^a$ is the word of length $M$ for which $a \in K_{w^a}$. In view of \eqref{pfur}, the corresponding replacement statement is that for all $\e' > 0$
\begin{equation} \label{pfuf}
  \limsup_{M,k,N} \Q_N \bigg( \sup_{0 \leq t \leq T} \bigg| \int_0^t \sum_{a \in V_0} \Big( \frac{3^N}{|V_N|}  \eta_s^N (a) \partial_N^\perp F_s(a) - \frac23 \pi_s^N(\iota_{w^a}^k) \partial^\perp F_s(a) \Big) \, ds \bigg| > \e' \bigg) = 0,
\end{equation}
where here and henceforth we use the short-hand notation
\begin{equation} \label{pfub}
  \limsup_{M,k,N} 
  := \limsup_{M \to \infty}  \limsup_{k \to \infty} \limsup_{N \to \infty},
\end{equation}
which demonstrates the importance of the ordering of $M,k,N$ in this notation. Due to this ordering, we may assume, whenever convenient, that $k$ is large enough with respect to $M$ and that $N$ is large enough with respect to $M$ and $k$. 

The proof of \eqref{pfuf} in \cite{CG} is split by proving the replacements of the three parts mentioned above one by one. The most involved step is the replacement of $\eta_s^N (a)$ by the approximate average $\pi_s^N(\iota_{w^a})$. The proof applies verbatim to our setting. However, since $\Q_N$ depends on $\cL_N$ and thus on the Glauber part, a comment is in order. The proof needs on $\cL_N$ only that the lower bound in \eqref{pfun} (already established) on the Dirichlet form holds. 

The other replacements that we apply to \eqref{pfug} are $\Delta_N^0$ by $\frac23 \Delta$,
\begin{equation} \label{pfue}
  \frac1{|V_N|} \sum_{x \in V_N^0} c_x(\eta_s^N) (1 - 2 \eta_s^N (x)) F_s(x)
  \quad \text{by} \quad
  \frac1{3^M} \sum_{|w| = M} \Phi(\pi_s^N(\iota_w)) \ov F_s(V_N^w),
\end{equation}
where $\Phi$ is defined in \eqref{Phi} and 
\begin{equation} \label{pftz}
  \ov F_s(V_N^w) := \frac1{|V_N^w|} \sum_{x \in V_N^w} F_s (x),
\end{equation}
$\ov F_s(V_N^w)$ by 
$
 \ov F_s^w := m(K_w)^{-1} \int_{K_w} F_s \, dm 
$
and $\pi_s^N(\iota_w)$ by $\pi_s^N(\iota_w^k)$.
The replacement in \eqref{pfue} is the main new part, which we treat in full detail in Section \ref{s:repl}. The replacement of $\ov F_s(V_N^w)$ by $\ov F_s^w$ is not done explicitly in \cite{CG}, but it can be easily proven from the continuity of $F$. Likewise, the replacement of $\pi_s^N(\iota_w)$ by $\pi_s^N(\iota_w^k)$ is easily justified from the Lipschitz continuity of $\Phi$ and \eqref{pfuc}.

Applying all these replacements to \eqref{pfug} yields that $M_t^N(F)$ in \eqref{pfur} can be replaced (possibly by multiplying $\e$ by a universal prefactor) by
\begin{multline*}
  \pi_t^N(F_t) 
- \pi_0^N(F_0) 
- \int_0^t \pi_s^N \Big( \Big( \partial_s + \frac23 \Delta \Big) F_s \Big) \, ds \\
- \frac1{3^M} \sum_{|w| = M} \int_0^t  \Phi(\pi_s^N(\iota_w^k)) \ov F_s^w \, ds 
+ \frac23 \sum_{a \in V_0} \int_0^t \pi_s^N(\iota_{w^a}^k) \partial^\perp F_s(a) \, ds =: \Psi(\pi_\bullet^N),
\end{multline*} 
where $\Psi$ is bounded and continuous in the Skorokhod topology. Hence, we can then pass to the limit $N \to \infty$ in the resulting form of \eqref{pfur} as
\[
\limsup_{M,k} \lim_{N \to \infty} \Q_N \Big( \sup_{0 \leq t \leq T} |\Psi(\pi_\bullet^N)| > C \e \Big) 
\geq \limsup_{M,k} \Q \Big( \sup_{0 \leq t \leq T} |\Psi(\pi_\bullet)| > C \e \Big).
\]
Since $\Q$ is concentrated on paths $\pi$ with a density $\rho \in L^2(0,T; \cF)$, we have \comm{[Cty in the Skor top: By the orange stuff above, it is sufficient, for the G term, to show that $\pi \mapsto \Phi(\pi(\iota_w^k))$ is bdd and cts with respect to the weak top of prob measures. Bdd is trivial. And, if $\pi_\e \weakto \pi$, then $\pi_\e(\iota_w^k) \to \pi(\iota_w^k)$ by defn (cause $\iota_w^k$ is a continuous fct), and the cty of $\Phi$ does the rest.]} 
\begin{multline} \label{pfud} 
    \Psi(\pi_\bullet) = \int_K \rho_t F_t \, dm 
      - \int_K \rho_0 F_0 \, dm
      - \int_0^t \int_K \rho_s \Big( \partial_s + \frac23 \Delta \Big) F_s \, dm ds 
      \\
      - \frac1{3^M} \sum_{|w| = M} \int_0^t  \Phi(\ov \rho_s^{w,k}) \ov F_s^w \, ds  
      + \frac23 \sum_{a \in V_0} \int_0^t \ov \rho_s^{w^a,k} \partial^\perp F_s(a)  \, ds,
\end{multline}
where $\ov \rho_s^{w,k} := \int_K \rho_s \iota_w^k \, dm$. 
\comm{[reminder: the goal is to replace \eqref{pfud} by 
\begin{multline*} 
    \Theta_t(\rho, F) = \int_K \rho_t F_t \, dm 
      - \int_K \rho_\circ F_0 \, dm
      - \int_0^t \int_K \rho_s \Big( \frac23 \Delta + \partial_s \Big) F_s \, dm ds 
      \\
      - \int_0^t \int_K \Phi(\rho_s) F_s \, dm ds
      + \frac23 \int_0^t \sum_{a \in V_0} \rho_B(a) \partial^\perp F_s(a)  \, ds.]
\end{multline*}
}

We replace \eqref{pfud} under $\Q$ by $\Theta_t(\rho, F)$ in four replacement steps: $\rho_0$ by $\rho_\circ$, $\ov \rho_s^{w^a}$ by $\rho_B(a)$, and two other replacements for the Glauber term. Since $\Theta_t(\rho, F)$ is independent of $M$ and $k$, this will complete the proof of \eqref{pfzw}. The first two replacements are done in \cite{CG}; the proof is rather simple, does not rely on the details of $\Q$, and relies on \eqref{pfuc} (with $\pi_s$ instead of $\pi_s^N$) and the continuity of $\rho_s$ on $K$ for a.e.\ $s$ for the replacement by $\rho_B(a)$. 
For the Glauber term, we first replace $\Phi(\ov \rho_s^{w,k})$ by $m(K_w)^{-1} \int_{K_w} \Phi(\rho_s(z)) \, dm(z)$. To see that this is possible, imagine \eqref{pfuf} for these two terms with $\limsup_{M,k} \Q$. We estimate \comm{[For the two CG replacements, the first is easy. The second is stated for $b \geq \frac53$ in (6.54), dealt with in (6.55) and (6.46)]}
\begin{align*}
  \bigg| \Phi(\ov \rho_s^{w,k}) - \frac1{m(K_w)} \int_{K_w} \Phi(\rho_s(z)) \, dm(z) \bigg|
  &\leq \frac1{m(K_w)} \int_{K_w} | \Phi(\ov \rho_s^{w,k}) -  \Phi(\rho_s) | \, dm \\
  &\leq \frac{L_\Phi}{m(K_w)} \int_{K_w} | \ov \rho_s^{w,k} - \rho_s | \, dm \\
  &\leq L_\Phi | \ov \rho_s^{w,k} - \ov \rho_s^w | 
  + \frac{L_\Phi}{m(K_w)} \int_{K_w} | \ov \rho_s^w - \rho_s | \, dm,
\end{align*}
where $L_\Phi$ is the Lipschitz constant of $\Phi$ and $\ov \rho_s^{w} := \int_K \rho_s \iota_w \, dm$. Then, for the replacement statement, for the two terms in the right-hand side, we rely on $\rho \in L^2(0,T; \cF)$; we bound the first term by applying \eqref{pfuc} and the second term by using $\cF \subset C(K)$. Thus, the replacement of the Glauber term in \eqref{pfud} reads
\begin{equation*}
  \frac1{3^M} \sum_{|w| = M} \int_0^t \bigg( \frac1{m(K_w)} \int_{K_w} \Phi(\rho_s) \, dm \bigg) \ov F_s^w \, ds  
  = \int_0^t \sum_{|w| = M} \int_{K_w} \Phi(\rho_s(z)) \ov F_s^w \, m(dz)ds.
\end{equation*}
Finally, we replace $\ov F_s^w$ by $F_s(z)$, which is easy to justify by using the regularity of the test function $F \in \cD_T$. This completes the proof of \eqref{pfzw}.

\subsection{Proof of Theorem \ref{t:HDL:un}}\label{s:uniqueness}

In this section, we provide a proof of Theorem \ref{t:HDL:un}.

\begin{proof}[Proof of Theorem \ref{t:HDL:un}] \cite[Section 7.2]{CG} contains the proof for $\Phi = 0$. We summarize the proof here to extend it to general $\Phi$ below. Assume that two weak solutions $\rho^0$ and $\rho^1$ exist. Based on \eqref{weak-form}, $\rho := \rho^0 - \rho^1$ satisfies
\begin{equation} \label{pfzy}
  0 = \int_K \rho_t F_t \, dm 
      + \int_0^t \int_K -\rho_s \partial_s F_s \, dm ds 
      + \frac23 \int_0^t \cE ( \rho_s, F_s) \, ds
\end{equation}
for all $t$ and all test functions $F$. We fix $t \in (0,T]$. There exists a sequence of test functions $\rho^n \in \cD_t$ that converges to $\rho$ in $L^2(0,t; \cF)$ as $n \to \infty$. We set $F_n(s,x) := \int_s^t \rho_r^n(x) \, dr$ for any $s \in [0,t]$ and any $x \in K$. Using $F_n$ as the test function in \eqref{pfzy}, the first integral is $0$, and the latter two converge as $n \to \infty$ to $R(t) := \int_0^t \int_K \rho_s^2 \, dm ds \geq 0$ and $\frac13 \cE(\int_0^t \rho_s \, ds) \geq 0$, respectively. Hence, both limits have to be $0$. This implies $\rho=0$ and thus $\rho^0 = \rho^1$.

For a general $\Phi$, we repeat the same construction. Then, in \eqref{pfzy}, the additional term
\begin{equation*}
  - \int_0^t \int_K \big( \Phi(\rho_s^0) - \Phi(\rho_s^1) \big) F_s \, dm ds
\end{equation*}
appears on the right-hand side. Applying the test function $F_n$ as above and taking $n \to \infty$, we obtain
\begin{equation*}
  R(t) 
  = \int_0^t \int_K \rho_t^2 \, dm ds
  = - \frac13 \cE \bigg( \int_0^t \rho_s \, ds \bigg) + \int_0^t \int_K \big( \Phi(\rho_s^0) - \Phi(\rho_s^1) \big) \int_s^t \rho_r \, dr dm ds
\end{equation*}
for all $t \in (0,T]$. We estimate the right-hand side from above,  using the Lipschitz constant $L_\Phi$ of $\Phi$, as
\begin{align*}
  \int_K \int_0^t \int_0^t  \big| \Phi(\rho_s^0) - \Phi(\rho_s^1) \big| |\rho_r| \,  dr  ds dm
  &\leq L_\Phi \int_K \bigg( \int_0^t |\rho_s| \, ds \bigg)^2 dm \\
  &\leq t L_\Phi \int_K \int_0^t \rho_s^2 \, ds dm
  = t L_\Phi R(t).
\end{align*}
We conclude that $R(t) \leq t L_\Phi R(t)$ for all $t \in (0,T]$. Since $R$ is nonnegative, we obtain $R(t) = 0$ for all $t \in [0, \frac1{L_\Phi}]$, and thus $\rho(t,x) = 0$ for all $x \in K$ and a.e.\ $t \in (0,\frac1{L_\Phi})$. Then, $(t,x) \mapsto \rho(t + \frac1{L_\Phi}, x)$ is a weak solution on the time interval $(0, T - \frac1{L_\Phi})$. Iterating the construction above, we obtain after $\lceil L_\Phi T \rceil$ steps that $\rho = 0$ in $L^2(0,T; \cF)$. 
\end{proof}

\section{Replacement of the Glauber term}
\label{s:repl}
 
This section justifies the replacement of the Glauber term, which is mentioned formally in \eqref{pfue}. The precise statement is given in Lemma \ref{l:replacement} below. It uses  
$\cD_T$, $V_N^w$, $\iota_w$ and $\ov F_s(V_N^w)$ introduced in 
\eqref{weak-form}, \eqref{pfuh}, \eqref{pfua} and \eqref{pftz},
the short-hand notation for limsup introduced in \eqref{pfub}, 
and the Markov inequality $\Q_N(|\Psi(\pi_\bullet^N)| > \e) \leq \frac1\e \E_{\mu_N} [|\Psi(\pi_\bullet^N)|]$ for any measurable function $\Psi$ on the path space. 

\begin{lem}[Main replacement] \label{l:replacement}
Let $T > 0$ and $F \in \cD_T$. Then
\begin{align*} 
  \limsup_{M,N} \E_{\mu_N} \bigg[ \sup_{0 \leq t \leq T} \bigg| \int_0^t U_{N,M}(\eta_s^N, F_s) \, ds \bigg|\bigg] = 0,
\end{align*}
where
\begin{align*}
U_{N,M}(\eta_s^N, F_s):=
\frac1{|V_N|} \sum_{x \in V_N^0} c_x(\eta_s^N) (1 - 2 \eta_s^N (x)) F_s(x) 
  - \frac1{3^M} \sum_{|w| = M} \Phi(\pi_s^N(\iota_w)) \ov F_s(V_N^w).
\end{align*}
\end{lem}

The proof of Lemma \ref{l:replacement} is the main technical novelty of this paper. 
The remainder of this section is devoted to the proof.

\subsection{First preparation: a new choice of $\varrho$} 

For technical reasons (Lemma \ref{l:mov:part} below, the so-called Moving Particle lemma), we do not work with the harmonic function $\varrho$ from Section \ref{s:pf:varrCG}. Instead, we introduce a different function $\varrho$ that still satisfies the conditions in \eqref{pfum} given by $\varrho \in \cF$ and $\varrho |_{V_0} = \rho_B$, but which is constant in most of the interior of $K$, except on a vanishing area around the three boundary sites. 

First, we introduce notation. Let $1 \leq N_B \leq N-1$ be an integer. We define the closed subsets of $K$ given by 
$$
  K^B := \bigcup_{a \in V_0} K_{v^a}, \qquad K^I := \ov{K \setminus K^B},
$$  
where $v^a$ is characterized by $|v^a| = N_B$ and $a \in K_{v^a}$. Note that $K^B \cap K^I$ is a finite set of 6 elements, which are precisely those sites of $V_{N_B}$ that are connected on $V_{N_B}$ to some boundary point $a \in V_0$.
 We interpret $K^B$ as the (thick) `boundary' part of $K$ and $K^I$ as the (reduced) `interior' part. We send $N_B$ to $\infty$ at the very end, and thus we may assume that it is much smaller than $N$ and $M$. 
We further set
\begin{equation*}
  V_N^I := V_N \cap K^I, \qquad \Omega_N^I := \{0,1\}^{V_N^I} 
\end{equation*}
to be the lattice with the boundary areas removed and the set of particle configurations on $V_N^I$, respectively. 

Recall the terminology and properties of harmonic functions stated in Section \ref{s:HDL:Eqn:calc}. We take $\varrho$ as the $N_B$-harmonic extension of
\begin{equation} \label{pfty} 
  \varrho(x)
  := \left\{ \begin{aligned}
    &\rho_B(x)
    &&x \in V_0 \\
    &\rho
    &&x \in V_{N_B}^0,
  \end{aligned} \right.
  \qquad \rho := \frac13 \sum_{a \in V_0} \rho_B(a).
\end{equation}
Then, $\varrho \in \cF$, i.e.\ it satisfies \eqref{pfum} as required. Moreover,
$ 
  \rho_* \leq \varrho \leq \rho^*
$
on $K$, where we recall from \eqref{pfum} that $0 < \rho_* \leq \rho^* < 1$ are the minimum and the maximum of $\rho_B$. Also, $\varrho$ is constant on $K^I$. Furthermore, $\cE_N(\varrho)$ is independent of $N$ for $N > N_B$. The precise choice of $\rho \in [\rho_*, \rho^*]$ is not important; the proof holds for any fixed value in $[\rho_*, \rho^*]$. We  keep $\varrho$ fixed in the remainder of Section \ref{s:repl}.  

With this new choice of $\varrho$, the lower bound on the Dirichlet form in Section \ref{s:pf:varrCG} changes. Indeed, while $\cE_N(\varrho)$ is still constant in $N$, it is not bounded in $N_B$. However, using that $\varrho$ is constant on $K^I$ we get
\begin{equation*}
   2 \Big( \frac35 \Big)^N \cE_N(\varrho)
   = \sum_{x \in V_N} \sum_{\substack{ y \in V_N \\ y \sim x }} [\varrho(x) - \varrho(y)]^2   
   = \sum_{a \in V_0} \sum_{x \in \ov V_N^{v^a}} \sum_{\substack{ y \in \ov V_N^{v^a} \\ y \sim x }} [\varrho(x) - \varrho(y)]^2,
 \end{equation*}
where $\ov V_N^{v^a} := V_N \cap K_{v^a}$ (this is $V_N^{v^a}$ with the three boundary sites included).
Since $\varrho$ is $N_B$-harmonic and $\ov V_N^{v^a}$ is a translated, miniature version of $V_{N - N_B}$, we get
\begin{align*}
   \sum_{x \in \ov V_N^{v^a}} \sum_{\substack{ y \in \ov V_N^{v^a} \\ y \sim x }} [\varrho(x) - \varrho(y)]^2
   &= \sum_{x' \in V_{N - N_B}} \sum_{\substack{ y' \in V_{N - N_B} \\ y' \sim x' }} [\varrho(\varphi_{v^a}(x')) - \varrho(\varphi_{v^a}(y'))]^2 \\
   &= 2 \Big( \frac35 \Big)^{N - N_B} \cE_{N - N_B}(\varrho \circ \varphi_{v^a})
   \leq C \Big( \frac35 \Big)^{N - N_B}.
 \end{align*} 
Then, applying the inequalities from Section \ref{s:pf:Dform}, we get
\begin{equation} \label{pfvd}
  \cD_N(\psi^2)
  \geq \Gamma_N (\psi^2) - C \Big( \frac35 \Big)^{N - N_B} 
\end{equation}
for any density $\psi^2$ with respect to $\nu_\varrho$,
where we have abbreviated
\begin{equation*}
  \cD_N(\psi^2) := \cD_N(\psi^2, \varrho)
  \quad \text{and} \quad
  \Gamma_N (\psi^2) := \Gamma_N (\psi^2, \varrho)
\end{equation*}
with the understanding that $\varrho$ is fixed in the remainder of Section \ref{s:repl}.

\subsection{Second preparation: a Moving Particle lemma} 

The following Lemma is a particularized version of \cite[Lemma 5.1]{CG}, originally proved in \cite[Theorem 1.1]{Chen2}. It is a statement about the Sierpi\'nski gasket, and does not depend on the choice of the particle system. \comm{[Differences with CG: they do general edge weights $c_{xy}$, but we put them as $1$. They do a general graph, but we apply it with $V_N^I$ only. They allow for any $\rho \in [0,1]$. They sum over all edges once, whereas our sums count double (hence the 2 in $R_{zz'}$ and the $\frac12$ in the RHS)]}

\begin{lem}[Moving Particle] \label{l:mov:part}
For any $N > N_B \geq 1$, any $f : \Omega_N^I \to \R$ and any $z,z' \in V_N^I$ 
  \begin{equation} \label{mov:part}
     E_{\nu_\rho} \big[ \big( f \big( \eta^{zz'} \big) - f(\eta) \big)^2 \big]
     \leq \frac12 R_{zz'}(V_N^I) E_{\nu_\rho} \Big[ \sum_{x \in V_N^I} \sum_{\substack{ y \in V_N^I \\ y \sim x} } \big( f \big( \eta^{xy} \big) - f(\eta) \big)^2 \Big],
   \end{equation} 
   where 
   \begin{equation} \label{Rzz}
     R_{zz'}(V_N^I) := \sup \bigg\{ \frac{ 2(g(z) - g(z'))^2 }{ \sum_{x \in V_N^I} \sum_{y \in V_N^I, \, y \sim x } (g(x) - g(y))^2 } \, \bigg| \, g : V_N^I \to \R \bigg\}
   \end{equation} 
   is the effective resistance on $V_N^I$ between $z$ and $z'$.
\end{lem}

To apply \eqref{mov:part}, we need a lower bound on $R_{zz'}(V_N^I)$. In the context of \eqref{mov:part}, for any word $\omega$ with $N_B \leq |\omega| < N$ and $K_\omega \cap V_0 = \emptyset$, and any $z,z' \in K_\omega  \cap V_N^I$ 
  \begin{equation} \label{pfux}
    R_{zz'}(V_N^I) \leq C \Big( \frac53 \Big)^{N - |\omega|}
  \end{equation}
for some constant $C = C(\rho_*, \rho^*) > 0$.
We refer to \cite[Lemma 1.6.1]{Str} for a proof. While this lemma states \eqref{pfux} on $V_N$, the proof applies verbatim in our case on $V_N^I$. The reason is that the proof starts from \eqref{Rzz}, and bounds the denominator from below by replacing $V_N^I$ by the subset $K_\omega \cap V_N$.  
\comm{[own notes p.126]}

\subsection{Reduction to 1-block and 2-blocks estimates} 
\label{s:repl:to12block}

We prove Lemma \ref{l:replacement} in a reversed order; we show that the estimate is implied by a simpler one, and repeat this process of finding simpler, sufficient estimates consecutively until the resulting estimate can be proven directly by the law of large numbers. In each step we treat, as before, all variables in the limsup (for the moment these are $M, N$) as sufficiently large with regard to their ordering. Any generic constants $C$ that appear are positive and independent of the variables in the limsup.

For the first step, we introduce some notation. Recall the finite range $L_0$ of $c_x$ and $L(x,y)$ as the shortest path length between $x$ and $y$ on $V_N$. 
For any word $w$ with $|w| = M$, let
\begin{equation} \label{pfut} 
  \begin{aligned}
   V_N^{L_0,w} &:= \Big\{ x \in V_N^w : \min_{y \in V_M} L(x,y) \geq L_0 + 1 \Big\} \subset V_N^w, \\
  V_* &:= V_N^0 \setminus \left(\bigcup_{|w| = M} V_N^{L_0,w} \right).
\end{aligned}
\end{equation}
Since $|V_N^w \setminus V_N^{L_0,w}| \leq C$ , we have $|V_*| \leq C |V_M| \leq C' 3^M$. 
\medskip

\noindent
\textbf{Step 1.} Lemma \ref{l:replacement} is implied by 
\begin{equation} \label{pfzu} 
  \limsup_{M,N} \E_{\mu_N} \bigg[ \frac1{3^M} \sum_{|w| = M} \int_0^T W_w (\eta_s^N) \, ds \bigg] = 0,
\end{equation}
where 
\begin{align} \label{pfvn}
  W_w(\eta) &:= \bigg| \frac1{|V_N^w|} \sum_{x \in V_N^{L_0,w}} c_x(\eta) (1 - 2 \eta (x)) - \Phi(\ov \eta_w) \bigg|, \qquad \ov \eta_w := \frac1{|V_N^w|} \sum_{x \in V_N^w} \eta (x).
\end{align} 

The proof for this statement follows from several simple replacements and manipulations applied to the limsup in Lemma \ref{l:replacement}. First, since $F_s \in C(K)$ and $c_x(\eta)$ is uniformly bounded, we may replace $F_s(x)$ by $\ov F_s(V_N^w)$. Second, by the observation below \eqref{pfvm} and the Lipschitz continuity of $\Phi$, we may also replace $\pi_s^N(\iota_w)$ by $\ov \eta_w$. Third, we write the sum of $x$ over $V_N^0$ as the joint sum of all words $w$ with $|w| = M$ and all $x \in V_N^{L_0,w}$. This misses the points $x \in V^*$, but since $|V_*| \leq C' 3^M$, these points yield only a vanishing contribution to the sum. Fourth, we bring the absolute value signs inside the time integral and inside the sum over $w$, put $|\ov F_s(V_N^w)|$ outside, and bound it by the maximum of $F_s(z)$ over $s \in [0,T]$ and $z \in K$. Noting that the integrand of the time integral has become an increasing function of $t$, we conclude that \eqref{pfzu} implies Lemma \ref{l:replacement}.
\medskip

To state the next step, we introduce the nonhomogeneous product measure $\nu_\varrho$ on $\Omega_N$, which is characterized by its marginals
\[
  \nu_\varrho( \eta(x) = 1 ) = \varrho(x) \qquad \text{for all } x \in V_N. 
\]

\noindent
\textbf{Step 2.} We may reduce \eqref{pfzu} to 
\begin{equation} \label{pfzt} 
  \limsup_{N_B,M,N} \sup_f
  E_{\nu_\varrho} \bigg[ \frac1{3^M} \sum_{|w| = M} W_w(\eta) f(\eta) \bigg]
  = 0,
\end{equation}
where the supremum is taken over all densities $f:\Omega_N\to[0,\infty)$ with respect to $\nu_\varrho$ with 
\begin{equation} \label{pfzr}
  \Gamma_N (f) \leq C 2^{N_B} \Big( \frac35 \Big)^{N}.
\end{equation} 

To prove this,
we use the entropy inequality with constant $2^{N_B} |V_N|$ (chosen a posteriori) and reference measure $\nu_\varrho$ to get
\begin{align}\label{pfuw}
&\E_{\mu_N} \bigg[ \frac1{3^M} \sum_{|w| = M} \int_0^T W_w (\eta_s^N) \, ds \bigg] \notag \\
&\quad  \leq \frac{H(\mu_N | \nu_\varrho)}{2^{N_B} |V_N|} + \frac1{2^{N_B} |V_N|} \log \E_{\nu_\varrho} \bigg[ \exp \bigg( \frac{2^{N_B} |V_N|}{3^M} \sum_{|w| = M} \int_0^T W_w (\eta_s^N) \, ds \bigg) \bigg],
\end{align}
where 
\[
  H(\mu_N | \nu_\varrho) := E_{\nu_\varrho} \Big[ \frac{d \mu_N}{d \nu_\varrho} \log \frac{d \mu_N}{d \nu_\varrho} \Big]
\] 
is the relative entropy. 
Since the maximum and minimum of $\varrho$ are attained at the boundary values $\rho_B$, we have that $H(\mu_N | \nu_\varrho) \leq C(\rho_B) |V_N|$ uniformly in $N_B, \mu_N$. 
To bound the second term in the right-hand side of \eqref{pfuw} we apply the Feynman-Kac formula \cite[Appendix 1, Lemma 7.2]{KL} to obtain 
\comm{[pf of $H$-bound: $f$ is a density with respect to $\nu_\rho$, so $f(\eta) \leq \max_\zeta \nu_\rho(\zeta)^{-1} = C_\rho^{|V_N|}$. Then, estimate $\log f$ like this, so that $E_{\nu_\rho} [ f] = 1$ remains in $H_N(f)$]}
\comm{[KL notes Sec 32.1]}
\begin{align}\label{pfuv}
&\frac1{2^{N_B} |V_N|} \log \E_{\nu_\varrho} \bigg[ \exp \bigg( \frac{2^{N_B} |V_N|}{3^M} \sum_{|w| = M} \int_0^T W_w (\eta_s^N) \, ds \bigg) \bigg] \notag\\
&\quad \leq  T \sup_f \bigg( E_{\nu_\varrho} \bigg[ \frac1{3^M} \sum_{|w| = M} W_w (\eta) f(\eta) \bigg] - \frac{5^N}{2^{N_B} |V_N|} \cD_N(f) \bigg),
\end{align}
where the supremum is taken over all densities $f:\Omega_N\to[0,\infty)$ with respect to $\nu_\varrho$ and $\cD_N(f)$ is the Dirichlet form of $\cL_N$. Since $W_w$ is uniformly bounded, the expectation in the right-hand side in \eqref{pfuv} is also uniformly bounded, and thus we may restrict the class of $f$ to satisfy $\cD_N(f) \leq C 2^{N_B} (\frac35)^N$. Using \eqref{pfvd}, we may weaken this bound to \eqref{pfzr},
and bound the Dirichlet form in \eqref{pfuv} as
\begin{equation*}
  - \frac{5^N}{2^{N_B} |V_N|} \cD_N(f) \leq C \Big( \frac56 \Big)^{N_B}.
\end{equation*}
Collecting all bounds obtained so far, we obtain
\begin{align*}
  \E_{\mu_N} \bigg[ \frac1{3^M} \sum_{|w| = M} \int_0^T W_w (\eta_s^N) \, ds \bigg]
  \leq C \Big( \frac56 \Big)^{N_B} + T \sup_f E_{\nu_\varrho} \bigg[ \frac1{3^M} \sum_{|w| = M} W_w (\eta) f(\eta) \bigg].
\end{align*}
Thus, taking $\limsup_{N_B,M,N}$ on both sides proves the claim in step 2.
\comm{\textbf{Rationale behind the next 1- and 2-blocks stuff}. The problem with proving \eqref{pfzt} directly is that the function $W_w$ depends on $N$, because while the $j$-cell $K_w$ is fixed, the grid $V_N^w$ grows finer in $N$. The fix is to use words $\omega$ of length $N-m$, as then $V_N^\omega$ stays constant in size, which is of $O(3^m)$. }  
\medskip

\noindent
\textbf{Step 3.} Reduction to $V_N^I$ and constant $\rho = \varrho |_{K^I}$ (recall \eqref{pfty}).  \eqref{pfzt} holds if
\begin{equation} \label{pfuz}
  \limsup_{N_B,M,N} \sup_f
  E_{\nu_\rho} \bigg[ \frac1{3^M} \sum_{w \in \cW_M^I} W_w(\eta) f(\eta) \bigg]
  = 0,
\end{equation}
where $\nu_\rho$ is the product measure on $\Omega_N^I := \{0,1\}^{V_N^I}$ with constant density $\rho = \varrho |_{K^I}$, 
\begin{equation*}
  \cW_M^I := \big\{ \text{words } w : |w| = M, \, K_w \subset K^I \big\},
\end{equation*}
and the supremum is taken over all densities $f:\Omega_N^I\to[0,\infty)$ with respect to $\nu_\rho$ which satisfy
\begin{equation} \label{pfva}
  \Gamma_N^I(f)
  := 
  \frac12 E_{\nu_\rho} \Big[ \sum_{x \in V_N^I} \sum_{ \substack{ y \in V_N^I \\ y \sim x } } \big[ \sqrt{ f(\eta^{xy}) } - \sqrt{ f(\eta) } \big]^2 \Big]
  \leq C 2^{N_B} \Big( \frac35 \Big)^N.
\end{equation}

To prove this, 
consider the expectation in \eqref{pfzt} for arbitrary admissible $f$.
We split the sum over $w$ into $w \in \cW_M^I$ and the remaining words $ \cW_M^B := \{ \text{words } w : |w| = M, \, K_w \subset K^B \}$. For $w \in \cW_M^B$, we simply apply the uniform estimate $W_w(\eta) \leq C$. Recalling that $f$ is a density with respect to $\nu_\varrho$ and noting that $|\cW_M^B| = 3^{M - N_B + 1}$, we bound  the expectation in \eqref{pfzt} from above by
\begin{align} \label{pfvb}
  \frac C{3^{N_B}} + \sum_{\eta \in \Omega_N} \frac1{3^M} \sum_{w \in \cW_M^I} W_w(\eta) f(\eta) \nu_\varrho(\eta).
\end{align}
We split $\eta = (\eta^I, \eta^B)$, where $\eta^I = \eta |_{V_N^I}$, $\eta^B = \eta |_{V_N^B}$ and $V_N^B := V_N \setminus V_N^I$. Note that for any $w \in \cW_M^I$, $W_w$ only depends on $\eta$ through $\eta^I$. Hence, \eqref{pfvb} reads as
\begin{align*} 
  \frac C{3^{N_B}} + \sum_{\eta^I \in \Omega_N^I} \frac1{3^M} \sum_{w \in \cW_M^I} W_w(\eta^I) \ov f(\eta^I) \nu_\varrho(\eta^I),
\end{align*}
where $\nu_\varrho(\eta^I) = \nu_\rho(\eta^I)$ by construction of $\varrho$, and
\begin{equation*}
  \ov f(\eta^I) := \sum_{\eta^B \in \Omega_N^B } f(\eta^I, \eta^B) \nu_\varrho(\eta^B), \qquad \Omega_N^B := \{0,1\}^{V_N^B}
\end{equation*} 
is a marginal of $f$. Hence, to complete the proof of the claim in Step 3, it is left to show that $\ov f$ satisfies the bound in \eqref{pfva}. This follows from the convexity of the carr\'e du champ operator (setting $\psi^2 := f$):
\begin{align*}
  \Gamma_N^I (\ov f)
  &\leq \sum_{\eta^B}  \Gamma_N^I \big( f(\cdot, \eta^B) \big) \nu_\varrho(\eta^B) \\
  &= \frac12 \sum_{\eta^B} \sum_{\eta^I} \Big[ \sum_{x \in V_N^I} \sum_{ \substack{ y \in V_N^I \\ y \sim x } } [\psi((\eta^I)^{xy}, \eta^B) - \psi(\eta^I, \eta^B)]^2 \Big] \nu_\rho(\eta^I) \nu_\varrho(\eta^B) \\
  &= \frac12 \sum_\eta \Big[ \sum_{x \in V_N^I} \sum_{ \substack{ y \in V_N^I \\ y \sim x } } [\psi(\eta^{xy}) - \psi(\eta)]^2 \Big] \nu_\varrho(\eta)
  \leq \Gamma_N (f) \leq C 2^{N_B} \Big( \frac35 \Big)^N. 
\end{align*}
\medskip

In the next step, we split \eqref{pfuz} into a 1-block and a 2-blocks estimate. On the Sierpi\'nski triangle, the word `block' is better replaced by `triangle', but we stick with the common terminology. We interpret the $M$-cells $K_w$ as the large blocks (small on the macroscopic scale) and, for $m \in \N$, the $(N-m)$-cells $K_\omega$ as the small blocks (large on the microscopic scale), where the words $\omega$ are of length $|\omega| = N-m > M = |w|$. Note that as $N \to \infty$, the large block $K_w$ remains unaltered, whereas the small block $K_\omega$ reduces in size, while the large discretized block $V_N^w$ gets finer (i.e.\ contains more sites), whereas the small discretized block $V_N^\omega$ remains a rescaled and translated copy of the $N$-independent $V_m^0$. We will occasionally decompose a word $\omega$ as the concatenation $wv$ of the words $w$ and $v$ of sizes $|w| = M$ and $|v| = N-M-m$. We will reserve the symbols $w,v,\omega$ for words of the particular aforementioned sizes.
\medskip

\noindent
\textbf{Step 4.} \eqref{pfuz} follows from to the 1-block and 2-blocks estimates
\begin{equation} \label{pfxx}
  \limsup_{N_B, m, N} \sup_f
  E_{\nu_\rho} \bigg[ \frac1{3^{N-m}} \sum_{\omega \in \cW_{N-m}^I} W_\omega(\eta) f(\eta) \bigg]
  = 0,
\end{equation}
\begin{equation} \label{pfxw}
  \limsup_{N_B, m, M, N} \sup_f 
  E_{\nu_\rho} \bigg[ \frac1{3^M} \sum_{w \in \cW_M^I} \frac1{3^{2k}} \sum_{|v| = k} \sum_{|v'| = k}  f(\eta) |\ov \eta_{wv} - \ov \eta_{wv'}| \bigg]
  = 0,
\end{equation} 
where $k := N-M-m$, $\nu_\rho$ and the supremum is as in Step 3, and $\ov \eta_\omega$ is the average of $\eta$ over $V_N^\omega$ defined in \eqref{pfvn}.
\medskip

To prove this, 
let $w \in \cW_M^I$. First, similar to the procedure in the proof of Step 1, we split off from the sum in $W_w$ over $x \in V_N^{L_0,w}$ (see \eqref{pfvn}) the sites around $V_{N-m}$. Analogously to \eqref{pfut} we set
 \begin{equation*}
   V_N^{L_0,\omega} = \Big\{ x \in V_N^\omega : \min_{y \in V_{N-m}} L(x,y) \geq L_0 + 1 \Big\}
 \end{equation*}  
and $V_*^w := V_N^{L_0,w} \setminus \cup_{|v| = k} V_N^{L_0,wv}$. Setting $g_x(\eta) := c_x(\eta) (1 - 2 \eta (x))$, we rewrite the sum in $W_w$ as  
\begin{align*}
  \frac1{|V_N^w|} \sum_{x \in V_N^{L_0,w}} c_x(\eta) (1 - 2 \eta (x)) 
  = \frac1{|V_N^w|} \sum_{|v| = k} \sum_{x \in V_N^{L_0,wv}} g_x(\eta)
    + \frac1{|V_N^w|} \sum_{x \in V_*^w} g_x(\eta).
\end{align*}
Since $g_x(\eta)$ is uniformly bounded, the second term in the right-hand side of the last expression is bounded in absolute value by $C |V_{N-m}^w| / |V_N^w| \leq C' 3^{-m}$, which vanishes in the limit $m\to\infty$.

Then, adding and subtracting averages over $V_N^{wv}$, we obtain
\begin{align*} 
  W_w(\eta)  
  &\leq 
  \frac C{3^m} + \bigg| \frac1{|V_N^w|} \sum_{|v| = k} \sum_{x \in V_N^{L_0,wv}} g_x(\eta) - \frac1{3^k} \sum_{|v| = k} \Phi(\ov \eta_{wv}) \bigg| + \bigg| \frac1{3^k} \sum_{|v| = k} \Phi(\ov \eta_{wv}) - \Phi(\ov \eta_w) \bigg| \\ 
 &\leq \frac C{3^m} + \frac1{3^k} \sum_{|v| = k} \bigg| \frac{3^k}{|V_N^w|} \sum_{x \in V_N^{L_0,wv}} g_x(\eta) -  \Phi(\ov \eta_{wv}) \bigg|
   + \frac1{3^k} \sum_{|v| = k} \big| \Phi(\ov \eta_{wv}) - \Phi(\ov \eta_w) \big|. \notag
\end{align*}
Noting that  
\[
  \frac{3^k |V_N^{wv}|}{|V_N^w|} = 3^k \frac{3^m - 1}{3^{N-M} - 1} = \frac{1 - 3^{-m}}{1 - 3^{M-N}},
\]
we obtain
\begin{align*}
  W_w(\eta)  
  &\leq \frac C{3^m} + \frac1{3^k} \sum_{|v| = k} W_{wv}(\eta)
   + \frac{L_\Phi}{3^k} \sum_{|v| = k} \big| \ov \eta_{wv} - \ov \eta_w \big| \\
 &\leq \frac{C'}{3^m} + \frac1{3^k} \sum_{|v| = k} W_{wv}(\eta)
   + \frac{L_\Phi}{3^{2k}} \sum_{|v| = k} \sum_{|v'| = k} \big| \ov \eta_{wv} - \ov \eta_{wv'} \big|, 
\end{align*}
where $L_\Phi$ is the Lipschitz constant of $\Phi$.
Applying this estimate to \eqref{pfuz} completes the proof of the claim in Step 4.
\medskip

In the remainder, we prove \eqref{pfxx} and \eqref{pfxw} in the following two subsections. Both proofs are inspired by \cite[Sections 5.4 and 5.5]{KL}. A major difference with those proofs is that there is no simple translation invariance on $V_N^I$.

\subsection{Proof of the 1-block estimate \eqref{pfxx}}
\label{s:repl:1block}

We relieve notation by writing $\eta_x := \eta(x)$ whenever convenient. Henceforth, the dependence of the Glauber rates $c_x$ and the related path-length function $L$ on $N$ becomes important; we denote them as $c_x^N$ and $L^N$.
\medskip

\noindent  
\textbf{Step 1}: Reduction to a generic copy $V_m^0$. \eqref{pfxx} holds if 
\begin{equation} \label{pfxr}  
  \limsup_{N_B,m,N} 
  \sup_f 
  E_{\nu_\rho} [ U(\xi) f (\xi) ]
  = 0,
\end{equation} 
where 
\[
  U(\xi) := \bigg| \frac1{|V_m^0|} \sum_{y \in V_m^{L_0}} c_y^m(\xi) (1 - 2 \xi_y) - \Phi(\ov \xi) \bigg|, \qquad 
  \ov \xi := \frac1{|V_m^0|} \sum_{y \in V_m^0} \xi_y
\] 
has replaced the previous term $W_\omega(\eta)$, 
\[
  V_m^{L_0} := \Big\{ x \in V_m : \min_{a \in V_0} L^m(x,a) \geq L_0 + 1 \Big\}
\]
and the supremum is taken over all densities $f$ on $\Omega_m^0 := \{0,1\}^{V_m^0}$ with respect to $\nu_\rho$  that satisfy
\begin{equation*}
  \Gamma_m^0 (f) := \frac12 E_{\nu_\rho} \Big[ \sum_{x \in V_m^0} \sum_{ \substack{ y \in V_m^0 \\ y \sim x } } [\psi(\eta^{xy}) - \psi(\eta)]^2 \Big] \leq C \frac{2^{N_B} 3^m}{5^N}.
\end{equation*} 
\medskip
 
To prove this, we
first reveal some structure of $W_\omega(\eta)$; it only depends on $\eta \in \Omega_N^I$ through $\eta_\omega := \eta|_{V_N^\omega}$, and moreover $W_\omega(\eta)$ is, as a function of $\eta_\omega$, independent of $\omega$. To see this, recall that
\[
  W_\omega(\eta)
  = \bigg| \frac1{|V_N^\omega|} \sum_{x \in V_N^{L_0,\omega}} c_x^N(\eta) (1 - 2 \eta (x)) - \Phi(\ov \eta_\omega) \bigg|.
\] 

We deal with the components of $W_\omega(\eta)$ that depend on $\omega$ one by one. Since $V_N^\omega = \varphi_\omega(V_m^0)$, $V_N^\omega$ is a translated, miniature version of $V_m^0$. In particular, we have that $|V_N^\omega| = |V_m^0|$ is independent of $\omega, N$ and that the sum over $x \in V_N^{L_0,\omega}$ can be re-indexed as a sum over $y \in V_m^{L_0}$ with $x = \varphi_\omega(y)$. Then, for any $x \in V_N^\omega$, we have, setting $y = \varphi_\omega^{-1}(x) \in V_m^0$,
\begin{equation} \label{pfvl}
  \eta (x) 
  = \eta (\varphi_\omega(y))
  = \eta_\omega (\varphi_\omega(y))
  =: \xi(y;\omega),
\end{equation}
where $\xi(\cdot;\omega) \in \{0,1\}^{V_m^0}$ is essentially the same as $\eta_\omega$ since $V_N^\omega = \varphi_\omega(V_m^0)$. Finally, regarding $c_x^N(\eta)$, recall $\Lambda_x^N$ from \eqref{Lambda:x} (we add $N$ in the superscript to highlight the dependence on $N$) and \eqref{cx}. For all $x \in V_N^{L_0,\omega}$ we have $\Lambda_x^N \subset V_N^\omega$, and thus $\Lambda_y^m = \varphi_\omega^{-1}(\Lambda_x^N) \subset V_m^0$. In particular, $\Lambda_x^N$ and $\Lambda_y^m$ have the same shape $\Lambda \in \cS$. In addition, 
\begin{equation*}
  \eta |_{\Lambda_x^N} 
  = \eta_\omega |_{\Lambda_x^N} 
  = \xi(\cdot; \omega) |_{\Lambda_y^m}, 
\end{equation*}
where the last inequality is understood as functions on $\{0,1\}^\Lambda$. Finally,
\begin{align*}
  c_x^N(\eta)
  = c \big( \eta_\omega |_{\Lambda_x^N} ; \Lambda \big)
  = c \big( \xi(\cdot; \omega) |_{\Lambda_y^m} ; \Lambda \big)
  = c_y^m( \xi(\cdot; \omega) ).
\end{align*}
Identifying $\xi(\cdot; \omega)$ with $\eta_\omega$, we obtain $W_\omega(\eta) = U(\eta_\omega)$.
 
Next, we introduce an alternative to the usual translation mapping $\tau_x$ on the discrete torus $\T_N^d$. This is needed because the Sierpi\'nski gasket is not translation invariant. Instead, we use the fact that it is self-similar. In particular, $V_N^\omega = \varphi_\omega(V_m^0)$ for each $\omega \in \cW_{N-m}^I$. We use this to permute $\eta$ by shifting $\eta_\omega$ over the words in $\cW_{N-m}^I$. Precisely, we decompose $\eta$ into the components $\eta_\omega$ for all $\omega \in \cW_{N-m}^I$ and the remaining component
$
 \eta_* := \eta |_{V_{N - m}^0}.
$
To shift over the words, we label
\begin{align} \label{pfuy}
  \cW_{N-m}^I = \{\omega_0, \omega_1, \ldots, \omega_{J-1}\}, \qquad J := |\cW_{N-m}^I| = 3^{N-m} - 3^{N-m-N_B+1}. 
\end{align}
We define $\omega_i + \omega_j := \omega_{i+j}$, where the addition of indices is understood modulo $J$. We write the first word as $o := \omega_0$. It satisfies $\omega + o = \omega$ for all $\omega \in \cW_{N-m}^I$.
Finally, we introduce the index shift as the rotation permutation $\sigma_\omega$ given by
\begin{equation*}
  (\sigma_{\omega'} \eta)_\omega := \eta_{\omega + \omega'},
  \qquad (\sigma_{\omega'} \eta)_* := \eta_*
\end{equation*}
for all words $\omega, \omega' \in \cW_{N-m}^I$.

With $W_\omega(\eta) = U(\eta_\omega)$ proved and $\sigma_\omega$ introduced we turn to \eqref{pfxx}.
Since $\nu_\rho$ is product and homogeneous, it is invariant under rotation by $\sigma_\omega$.
Hence, we get with the change of variables $\zeta = \sigma_\omega \eta$ and with $|\cW_{N-m}^I| \leq 3^{N-m}$
\begin{align} \notag   
& E_{\nu_\rho} \bigg[ \frac1{3^{N-m}} \sum_{\omega \in \cW_{N-m}^I} W_\omega(\eta) f(\eta) \bigg] \\\notag
  &= \frac1{3^{N-m}} \sum_{\eta \in \Omega_N^I} \sum_{\omega \in \cW_{N-m}^I}  U(\eta_\omega) f(\eta) \nu_\rho(\eta) \\\notag
  &= \frac1{3^{N-m}} \sum_{\zeta \in \Omega_N^I} \sum_{\omega \in \cW_{N-m}^I}  U(\zeta_o) f(\sigma_{-\omega} \zeta) \nu_\rho(\zeta) \\\label{pfxf}
  &\leq E_{\nu_\rho} \big[ U(\eta_o) \ov f(\eta) \big],
  \qquad \ov f(\eta) := \frac1{|\cW_{N-m}^I|} \sum_{\omega \in \cW_{N-m}^I} f(\sigma_{-\omega} \eta),
\end{align}
where $\ov f$ is a density with respect to $\nu_\rho$.
Writing momentarily $\eta = (\eta_o, \eta') \in \Omega_N^I$, we have that integrating over $\eta'$ yields that
\begin{align} \label{pfxo}  
  E_{\nu_\rho} \big[ U(\eta_o) \ov f(\eta) \big]
  = E_{\nu_\rho} \big[ U(\eta_o) \ov f_o(\eta_o) \big],
  \qquad \ov f_o(\eta_o) := \sum_{\eta'} \ov f(\eta_o, \eta') \nu_\rho(\eta'),
\end{align}
where $\nu_\rho$ in the right-hand side is the restriction of $\nu_\rho$ to $V_N^o$, and $\ov f_o$ is a marginal of $\ov f$, which can also be interpreted as a weighted average over the densities $\ov f(\cdot, \eta')$ indexed over $\eta'$. Note that the word $o = \omega_0$ is arbitrarily picked from the dictionary $\cW_{N-m}^I$; it has no particular property over the other words. Recalling the interpretation of $\eta_o$ in \eqref{pfvl}, we observe that the expectation in \eqref{pfxo} equals that in \eqref{pfxr} for the particular choice $\ov f_o$ for the density. Therefore, it remains to show that $\ov f_o$ is admissible for any admissible $f$ in \eqref{pfxx}.
\comm{The configurations $\xi : V_m^0 \to \{0,1\}$ and $\eta_o : V_N^o \to \{0,1\}$ are identical up to a change of variable of the index. We use them interchangeably. KL applies a similar abuse of notation, so this seems acceptable.}
 
Let such $f$ be fixed.  
Since we have already shown that $\ov f_o$ is a density with respect to $\nu_\rho$ on $\Omega_m^0$, it is left to show that 
\begin{equation} \label{pfwo}
  \Gamma_m^0 (\ov f_o) \leq C \frac{2^{N_B} 3^m}{5^N}.
\end{equation}
Our strategy for proving \eqref{pfwo} is to estimate $\Gamma_m^0 (\ov f_o)$ by consecutive (in)equalities to arrive at an expression in terms of $\Gamma_N^I(f)$, for which we have the bound $C 2^{N_B} (\frac35)^N$ in \eqref{pfva}. 

Using the notation $\eta = (\eta_o, \eta') \in \Omega_N^I$ from \eqref{pfxo}, it follows from the convexity of $\Gamma_m^0$ that, writing $\ov \psi^2 := \ov f$ and $x_\omega := \varphi_\omega(x) \in V_N^\omega$ for $x \in V_m^0$,  \comm{[On notation below: alternative to $x_o, y_o$, we can also sum over $x', y' \in V_N^o$. The downside is then especially in the display further below, where the change of variables will become harder to understand]}
\begin{align}   \notag
  \Gamma_m^0 (\ov f_o)
  &\leq \sum_{\eta'} \Gamma_m^0 \big( \ov f(\cdot, \eta') \big) \nu_\rho (\eta') \\ \notag
  &= \frac12 \sum_{\eta'} \sum_{\xi \in \Omega_m^0} \sum_{x \in V_m^0} \sum_{y \sim x} \big[ \ov \psi \big( \xi^{x y}, \eta' \big) - \ov \psi(\xi, \eta') \big]^2 \nu_\rho(\xi) \nu_\rho (\eta') \\ \notag
  &= \frac12 \sum_{\eta'} \sum_{\eta_o \in \{0,1\}^{V_N^o} } \sum_{x \in V_m^0} \sum_{y \sim x} \big[ \ov \psi \big( (\eta_o)^{x_o y_o}, \eta' \big) - \ov \psi(\eta_o, \eta') \big]^2 \nu_\rho(\eta_o) \nu_\rho (\eta') \\\label{pfvj}
  &= \frac12 \sum_{\eta \in \Omega_N^I} \sum_{x \in V_m^0} \sum_{y \sim x} \big[ \ov \psi \big( \eta^{x_o y_o} \big) - \ov \psi(\eta) \big]^2 \nu_\rho(\eta)
  =: \Gamma_N^o (\ov f), 
\end{align}  
where $y \sim x$ is understood on $V_m^0$. 
Note that $\Gamma_N^o$ is the part of $\Gamma_N^I$ (recall \eqref{pfva}) that only accounts for neighboring sites in $V_N^o \subset V_N^I$. 

To proceed, we need the identity
\begin{equation} \label{pfvk}
  \sigma_{-\omega} \big( \eta^{x_o y_o} \big) = (\sigma_{-\omega} \eta)^{x_\omega y_\omega}
\end{equation}
for all $\omega \in \cW_{N-m}^I$ and all $x,y \in V_m^0$. While \eqref{pfvk} follows directly from the definitions, it requires some setup to verify it, and thus we provide the details. Note that for $\eta \in \Omega_N^I$, words $\omega, \omega'  \in \cW_{N-m}^I$ and $x,y,z \in V_m^0$:
\begin{itemize}
  \item $(\sigma_{-\omega} \eta)(z_{\omega'}) = \eta(z_{\omega' - \omega})$;
  \item $\displaystyle \eta^{x_\omega y_\omega}(z_{\omega'}) = \begin{cases}
    \eta(z_{\omega'})
    &\text{if } z_{\omega'} \notin \{ x_\omega,  y_\omega \} \\
    1- \eta(z_{\omega'})
    &\text{if } z_{\omega'} \in \{ x_\omega,  y_\omega \}.
  \end{cases}$
\end{itemize} 
Note that $z_{\omega'} \in \{ x_\omega,  y_\omega \} \iff \omega' = \omega$ and $z \in \{x,y\}$.  Assume $z_{\omega'} \notin \{ x_\omega,  y_\omega \}$.
Then, for the left-hand side in \eqref{pfvk}, take $\zeta := \eta^{x_o y_o}$ and expand
\[
  \big(\sigma_{-\omega} \big( \eta^{x_o y_o} \big)\big)(z_{\omega'})
  = (\sigma_{-\omega} \zeta) (z_{\omega'})
  = \zeta(z_{\omega' - \omega})
  = \eta(z_{\omega' - \omega}).
\]
For the right-hand side, take $\xi := \sigma_{-\omega} \eta$ and expand
\[
  (\sigma_{-\omega} \eta)^{x_\omega y_\omega} (z_{\omega'})
  = \xi^{x_\omega y_\omega} (z_{\omega'})
  = \xi (z_{\omega'})
  = \eta(z_{\omega' - \omega}).
\]
Hence, \eqref{pfvk} holds when $z_{\omega'} \notin \{ x_\omega,  y_\omega \}$. The proof is analogous in the case $z_{\omega'} \in \{ x_\omega,  y_\omega \}$.

We continue with the proof of \eqref{pfwo}.  
Using the convexity of $\Gamma_N^o$ and writing $(\sigma_\omega f) (\eta) := f(\sigma_\omega \eta)$ and $\psi^2 := f$, we continue the estimate above as
\begin{align} \notag
  \Gamma_N^o (\ov f) 
  &\leq  \frac1{|\cW_{N-m}^I|} \sum_{\omega \in \cW_{N-m}^I} \Gamma_N^o ( \sigma_{-\omega} f ) \\\notag
  &\leq  \frac2{3^{N-m}} \sum_{\omega \in \cW_{N-m}^I} \Gamma_N^o ( \sigma_{-\omega} f ) \\\notag
  &= \frac1{3^{N-m}}  \sum_{\eta \in \Omega_N^I} \sum_{\omega \in \cW_{N-m}^I} \sum_{x \in V_m^0} \sum_{y \sim x} \big[ \psi \big( \sigma_{-\omega} \big( \eta^{x_o y_o} \big) \big) - \psi(\sigma_{-\omega} \eta) \big]^2 \nu_\rho(\eta) \\\notag
  &= \frac1{3^{N-m}}  \sum_{\eta \in \Omega_N^I} \sum_{\omega \in \cW_{N-m}^I} \sum_{x \in V_m^0} \sum_{y \sim x} \big[ \psi \big( (\sigma_{-\omega} \eta)^{x_\omega y_\omega} \big) - \psi(\sigma_{-\omega} \eta) \big]^2 \nu_\rho(\eta) \\\label{pfvi} 
  &= \frac1{3^{N-m}}  \sum_{\zeta \in \Omega_N^I} \sum_{\omega \in \cW_{N-m}^I} \sum_{x \in V_m^0} \sum_{y \sim x} \big[ \psi \big( \zeta^{x_\omega y_\omega} \big) - \psi(\zeta) \big]^2 \nu_\rho(\zeta). 
\end{align}
Note that with $x' := x_\omega$ and $y' := y_\omega$ the triple sum over $\omega,x,y$ can be written as the double sum over $x' \in V_N^I \setminus V_{N-m}$ and $y' \in V_N^I \setminus V_{N-m}$ with $x' \sim y'$ on $V_N^I$. Then, the expression above equals $2 \cdot 3^{m-N} \Gamma_N^I(f)$ (recall \eqref{pfva}) except for some missing nonnegative contributions around the sites in $V_{N-m}^0$. Hence
\begin{equation*}
  \Gamma_m^0 (\ov f_o)
  \leq 2 \cdot 3^{m-N} \Gamma_N^I(f)
  \leq C \frac{2^{N_B} 3^m}{5^N}.
\end{equation*}
This completes the proof for the claim in Step 1.
\medskip

\noindent
\textbf{Step 2}: Taking $N \to \infty$. \eqref{pfxr} holds if
\begin{equation} \label{pfxl} 
  \limsup_{m\to\infty} \sup_{\Gamma_m^0(f) = 0}
  E_{\nu_\rho} [ U(\xi) f (\xi) ]
  = 0,
\end{equation}
where $f$ is a density with respect to $\nu_\rho$. 
\medskip

The proof starts with some observations. Since $\nu_\rho$ in \eqref{pfxr} is strictly positive as a measure on $\Omega_m^0$, and since $V_m^0$ is a finite set, the set $\cK_N$ of all densities of $f$ is compact. Moreover,
since $\nu_\rho$ in \eqref{pfxr} is independent of $N$,
$\cK_N \subset \cK$, where $\cK$ is a finite dimensional compact set.
Finally, both $\Gamma_m^0$ and the expectation  in \eqref{pfxr} are continuous as a function of $f$. 

From these three observations, we obtain that for each $N$ there exists a maximizer $f_N$. In addition, $(f_N)_N$ has at least one limit point. Let $f$ be any such point, and take $N_k$ such that $f_{N_k} \to f$. Then, $f$ is still a density with respect to $\nu_\rho$, $E_{\nu_\rho} [ U(\xi) f_{N_k} (\xi) ] \to E_{\nu_\rho} [ U(\xi) f (\xi) ]$ and $C 2^{N_B} 3^m 5^{-N_k}
\geq \Gamma_m^0(f_{N_k}) \to \Gamma_m^0(f)$ as $k \to \infty$, and thus $\Gamma_m^0(f) = 0$. From these results, it follows that the left-hand side in \eqref{pfxr} is bounded from above by the left-hand side in \eqref{pfxl}. The reason that $N_B$ has vanished from the limsup is that \eqref{pfxl} is independent of $N_B$.
\medskip

\noindent
\textbf{Step 3}: Fixing the total number of particles. \eqref{pfxl} holds if
\begin{equation} \label{pfxk}
  \limsup_{m\to\infty} \max_{0 \leq j \leq |V_m^0|}
  E_{\nu^j} \bigg[ \bigg| \frac1{|V_m^0|} \sum_{y \in V_m^{L_0}} c_y^m(\xi)(1 - 2\xi_y) - \Phi ( \ov \xi ) \bigg| \bigg]
  = 0,
\end{equation}
where $\nu^j$ is the restriction of $\nu_\rho$ to $\Omega_m^{0,j} := \{ \xi \in \Omega_m^0 : \ov \xi =  j|V_m^0|^{-1} \}$, i.e.\ the set of all configurations with $j$ particles. 
\medskip

To prove this, we observe from $\Gamma_m^0(f) = 0$ that $f$ is constant on $\Omega_m^{0,j}$ for each $j$. Hence, by conditioning $\nu_\rho$ on $j$ in \eqref{pfxl}, we can replace $f(\xi)$ by a constant $C_\rho$. Then, the supremum over $f$ reduces to the maximum over $j$. 
\medskip

\noindent
\textbf{Step 4}: Equivalence of ensembles. \eqref{pfxk} holds.

\medskip
To prove this, we fix $0 \leq j \leq |V_m^0|$. It is easy to check that $\nu^j$ is exchangeable. Then, by de Finetti's theorem, there exists $\lambda_j \in \cP([0,1])$ such that $\nu^j = \int \nu_\alpha \lambda_j(d\alpha)$. Thus, \eqref{pfxk} follows if \comm{[This paragraph is from Liggett85 p365 top. Own notes p.121 middle.]}
\begin{equation} \label{pfvg}
  \limsup_{m\to\infty} \max_{0 \leq \alpha \leq 1}
  E_{\nu_\alpha} \bigg[ \bigg| \frac1{|V_m^0|} \sum_{y \in V_m^{L_0}} c_y^m(\xi)(1 - 2\xi_y) - \Phi ( \ov \xi ) \bigg| \bigg]
  = 0.
\end{equation}

Next, we apply three consecutive replacements to the term $\Phi ( \ov \xi )$ in \eqref{pfvg}. First, it follows from a standard, direct computation that $E_{\nu_\alpha}[|\ov \xi - \alpha|^2] \leq |V_m^0|^{-1}$ for all $\alpha \in [0,1]$. Hence, using the fact that $\Phi$ is Lipschitz, we may replace $\Phi ( \ov \xi )$ by $\Phi ( \alpha )$. Second, recall from \eqref{Phi} that
\begin{align*}
   \Phi(\alpha)
   = \sum_{\Lambda \in \cS} \Big( \lim_{N \to \infty} r_\Lambda^N  \Big) E_{\nu_\alpha} \Big[(1 - 2\eta_0)  c(\eta; \Lambda) \Big],
\end{align*}  
where $\eta \in \{0,1\}^\Lambda$.
Since $|\cS| < \infty$ and since $\cS$ and $r_\Lambda^N$ are independent of $m,\alpha$, we may replace $r_\Lambda^N$ by $r_\Lambda^m$ for all $\Lambda \in \cS$. Thus, we may replace $\Phi(\alpha)$ by
\begin{equation*}
  \sum_{\Lambda \in \cS} r_\Lambda^m E_{\nu_\alpha} [c(\eta; \Lambda) (1 - 2\eta_0)  ]
  = \frac1{|V_m^0|} \sum_{y \in V_m^0}  E_{\nu_\alpha} [c_y^m(\xi)(1 - 2\xi_y)].
\end{equation*}
Third, since the contribution to this sum of $y$ over $V_m^0 \setminus V_m^{L_0}$ is negligible uniformly in $\alpha$, we may reduce the sum to $y$ over $V_m^{L_0}$. 
In conclusion, \eqref{pfvg} follows if 
\begin{equation} \label{pftu}
  \limsup_{m\to\infty} \max_{0 \leq \alpha \leq 1}
  E_{\nu_\alpha} \bigg[ \bigg| \frac1{|V_m^0|} \sum_{y \in V_m^{L_0}} ( \Xi_y - E_{\nu_\alpha}[\Xi_y] ) \bigg| \bigg]
  = 0,
  \qquad \Xi_y := c_y^m(\xi)(1 - 2\xi_y).
\end{equation}

Note that $\{ \Xi_y \}_{y \in V_m^{L_0}}$ is under $\nu_\alpha$ dependent due to the local dependence of $c_y^m(\xi)$ on $\xi$. To untangle these dependencies,
we divide $V_m^{L_0}$ into the smaller `triangles' $\{ V_m^{L_0,w} \}_{|w| = m-k}$ and the remaining set of sites $V_*$ (recall \eqref{pfut}), where we will take $\limsup_k$ on both sides in \eqref{pftu}. We decompose the sum in \eqref{pftu} into the sum over $y \in V_*$ and the double sum over $z \in V_k^{L_0}$ and $|w| = m-k$ (with $y = \varphi_w(z) =: z_w$). Since $|V_*| \leq C 3^{m-k}$, the sum over $y \in V_*$ has vanishing contribution uniformly in $\alpha$, and thus we may neglect it. Hence, \eqref{pftu} follows if
\begin{equation} \label{pfvf}
  \limsup_{k,m} \max_{0 \leq \alpha \leq 1}
  \frac1{|V_k^{L_0}|} \sum_{z \in V_k^{L_0}} E_{\nu_\alpha} \bigg[ \bigg| \frac1{3^{m-k}} \sum_{|w| = m-k} ( \Xi_{z_w} - E_{\nu_\alpha}[\Xi_{z_w}] ) \bigg| \bigg]
  = 0. 
\end{equation}
Note that $\Xi_{z_w}$ as indexed by $w$ are i.i.d.\ under $\nu_\alpha$.  Then, a standard computation yields (recall from \eqref{cinf} that $|\Xi_{z_w}| \leq \|c\|_\infty$)
\begin{equation*}
  E_{\nu_\alpha} \bigg[ \bigg| \frac1{3^{m-k}} \sum_{|w| = m-k} ( \Xi_{z_w} - E_{\nu_\alpha}[\Xi_{z_w}] ) \bigg|^2 \bigg]
  \leq \frac{ \| c \|_\infty^2 }{3^{m-k}}.
\end{equation*}
This implies \eqref{pfvf}. This completes the proof of the claim in Step 4.

\subsection{Proof of the 2-blocks estimate \eqref{pfxw}}

In the next step, we work on the set $V_m^0 \times \{1,2\}$ consisting of two copies of $V_m^0$. We write $(x,i) \in V_m^0 \times \{1,2\}$ as $x_i$, where $i \in \{1,2\}$ indicates the copy. For $x, y \in V_m^0$, $i \in \{1,2\}$ and configurations $\zeta : V_m^0 \times \{1,2\} \to \{0,1\}$, we have, in accordance with \eqref{eta:xy},
\begin{equation*}
  \zeta^{y_1 y_2}(x_i)
  = \left\{ \begin{aligned}
     &\zeta(y_2)
     &&\text{if } x_i=y_1 \\
     &\zeta(y_1)
     &&\text{if } x_i=y_2 \\
     &\zeta(x_i)
     &&\text{otherwise.}
   \end{aligned} \right.
\end{equation*}
Other than this formula, we can and will use lighter notation. We will write $x$ instead of $x_1$, $x'$ instead of $x_2$, and identify $\zeta$ with $(\xi, \xi')$ for $\xi, \xi' \in \Omega_m^0$. We write $\zeta^{y_1 y_2}$ as $(\xi, \xi')^{yy'}$. 
\medskip

\noindent 
\textbf{Step 1}: Reduction to two copies of $V_m^0$. Let $b \in V_m^0$. \eqref{pfxw} holds if 
\begin{equation} \label{pfxt}
  \limsup_{N_B, m, M, N} \sup_f E_{\nu_\rho^2} \big[ f(\xi, \xi') |\ov \xi - \ov \xi'| \big] = 0,
\end{equation}
where $\xi, \xi' \in \Omega_m^0$, $\nu_\rho^2$ is the product measure on $\Omega_m^0 \times \Omega_m^0$, and the supremum is taken over all densities $f:\Omega_m^0\times\Omega_m^0\to[0,\infty)$ with respect to $\nu_\rho^2$ such that $\Gamma_m^{2,b} (f) \leq C 2^{N_B} (\frac35)^{M}$, where  
\begin{align*} 
  \Gamma_m^{2,b} &:= \Gamma_m^2 + \Gamma_m^{2'} + \Gamma_m^b, \\ 
  \Gamma_m^2 (\psi^2)
  &:= \frac12 E_{\nu_\rho^2} \bigg[ \sum_{y \in V_m^0} \sum_{ z \sim y } 
         [\psi(\xi^{yz}, \xi') - \psi(\xi, \xi')]^2 \bigg],  \\
  \Gamma_m^{2'} (\psi^2) &:= \frac12 E_{\nu_\rho^2} \bigg[ \sum_{y \in V_m^0} \sum_{ z \sim y } 
         [\psi(\xi, (\xi')^{yz}) - \psi(\xi, \xi')]^2 \bigg], \\  
  \Gamma_m^b (\psi^2) &:=  \frac12 E_{\nu_\rho^2} \Big[
         [\psi((\xi, \xi')^{bb'}) - \psi(\xi, \xi')]^2   \Big].         
\end{align*} 
Note that  $\Gamma_m^2$ and $\Gamma_m^{2'}$ resemble $\Gamma_m^0$ from \eqref{pfxr}; the corresponding particle systems are the same, but $\Gamma_m^2$ describes in addition that the particles described by $\xi'$ do not move. Similarly, $\Gamma_m^{2'}$ keeps the particles described by $\xi$ fixed.
The sole purpose of the additional term $\Gamma_m^b$ is to have the property that $\Gamma_m^{2,b}(f) = 0$ implies that $f$ is constant on level sets of a fixed number of particles in $\Omega_m^0 \times \Omega_m^0$. $\Gamma_m^b$ is the carr\'e du champ operator of Kawasaki dynamics at site $b$ from either of the two copies of $V_m^0$ to the other.
\medskip
 
The proof of the claim in Step 1 is similar to that of Step 1 of the 1-block estimate in Section \ref{s:repl:1block}. On the one hand it is easier because the function $W_w(\eta)$ is now given by the simpler expression $|\ov \eta_{wv} - \ov \eta_{wv'}|$, but on the other hand it is a more delicate version because (a) we need to work with both large and small triangles, and (b) we need to handle two particle configurations $\xi$ and $\xi'$ that are connected by the new operator $\Gamma_m^b$. We are using the same notation from the proof of \eqref{pfxr} unless mentioned otherwise. 

Regarding (a), we decompose the index shift $\sigma_\omega$ into a shift over large triangles and a shift over small triangles. We do this by decomposing $\omega = wv$ as done above Step 4 in Section \ref{s:repl:to12block}, i.e.\ $w$ is always a word of length $M$ and $v$ is always a word of length $k := N - M - m$. For $w$, we use the restricted dictionary 
\[ 
  \cW_M^I = \{w_0, w_1, \ldots, w_{J-1}\}, \qquad J := |\cW_M^I| = 3^M - 3^{M-N_B+1} - 1,
\]
and for $v$ we use the complete dictionary
\[
  \cW_k := \{ v : |v| = k \}= \{v_0, v_1, \ldots, v_{K-1}\}, \qquad K := |\cW_k| = 3^k.
\]
Analogous to the addition of words defined below \eqref{pfuy}, we define the addition of $w, w' \in \cW_M^I$ and the addition of $v,v \in \cW_k$ based on the labelling in the displays above.
For all words $w \in \cW_M^I$ and all $v,v' \in \cW_k$, the shift over large triangles is given by 
\begin{align*}
  (\sigma_{w'o} \eta)_{wv} := \eta_{(w + w')v},
  \qquad (\sigma_{w'o} \eta)_* := \eta_*, \qquad o := v_0,
\end{align*}
where we recall that $\eta_* = \eta |_{V_{N-m}^0}$,
and the shift over small triangles is given by 
\comm{[
The reason for the symbol 'c' below is that it looks as a letter closest to `$o$']
}
\begin{align*}
  (\sigma_{c v'} \eta)_{wv} := \eta_{w(v + v')},
  \qquad (\sigma_{c v'} \eta)_* := \eta_*, \qquad c := w_0.
\end{align*} 
It is easy to check that $\sigma_{w o} \sigma_{c v} = \sigma_\omega = \sigma_{c v} \sigma_{w o}$  with $\omega = wv$. 
 
Consider the expectation in \eqref{pfxw}.
Note that $\nu_\rho$ is invariant under rotation by both $\sigma_{wo}$ and $\sigma_{c v}$.  
First substituting $\zeta = \sigma_{wo} \eta$ and then $\eta = \sigma_{c v} \zeta$ yields that the expectation in \eqref{pfxw} equals
\begin{align*}
  &\sum_{\zeta \in \Omega_N^I} \frac1{3^M} \frac1{3^{2k}} \sum_{w \in \cW_M^I} \sum_{v,v' \in \cW_k} f(\sigma_{(-w)o} \zeta) |\ov \zeta_{cv} - \ov \zeta_{cv'}| \nu_\rho(\zeta) \\
  &= \sum_{\eta \in \Omega_N^I} \frac1{3^M} \frac1{3^{2k}} \sum_{w \in \cW_M^I} \sum_{v,v' \in \cW_k} f(\sigma_{(-w)(-v)} \eta) |\ov \eta_{co} - \ov \eta_{c(v'-v)}| \nu_\rho(\eta).
\end{align*}
By shifting the sum over $v'$ to $v' - v$, the dependence of the summand on $w,v,v'$ decouples, and we obtain that the above display equals
\begin{align*} 
  \frac{|\cW_M^I|}{3^M} \sum_{\eta \in \Omega_N^I} \frac1{3^k} \sum_{v' \in \cW_k} \ov f(\eta) |\ov \eta_{co} - \ov \eta_{cv'}| \nu_\rho(\eta),
  \qquad \ov f(\eta) := \frac1{3^k |\cW_M^I|} \sum_{w \in \cW_M^I} \sum_{v \in \cW_k} f(\sigma_{(-w)(-v)} \eta).
\end{align*}
Note that the density $\ov f$ is the same as that in \eqref{pfxf}.
Clearly, the left expression above is nonnegative. We bound it from above by applying $|\cW_M^I| \leq 3^M$ and by moving the sum over $v'$ to the left and then taking the maximum over $v'$. Let $u \in \cW_k$ be a maximizer. Since $v'=o$ is a minimizer, we may assume $u \neq o$. Then, the expectation in \eqref{pfxw} is bounded from above by
\begin{align*} 
  \sum_{\eta \in \Omega_N^I} \ov f(\eta) |\ov \eta_{co} - \ov \eta_{cu}| \nu_\rho(\eta)
  = E_{\nu_\rho^2} \big[ \ov f_u(\eta_{co}, \eta_{cu}) |\ov \eta_{co} - \ov \eta_{cu}| \big],
\end{align*}
where
\begin{equation*}
  \ov f_u(\eta_{co}, \eta_{cu}) := \sum_{\tilde \eta} \ov f(\eta_{co}, \eta_{cu}, \tilde \eta) \nu_\rho(\tilde \eta)
\end{equation*}
and $\eta$ is decomposed as $\eta = (\eta_{co}, \eta_{cu}, \tilde \eta) \in \Omega_N^I$. Thus, $\ov f_u$ is a marginal of $\ov f$, which is therefore a density with respect to $\nu_\rho^2$. Similar to the proof of the 1-block estimate, we recognize this expression as the desired expectation in \eqref{pfxt}. 

It is left to show that $\Gamma_m^{2,b} (\ov f_u) \leq C 2^{N_B} ( \frac35 )^M$. We prove this by showing that
\begin{align} \label{pfwd}
  \Gamma_m^2 (\ov f_u) 
  &\leq C 2^{N_B} 3^m 5^{-N}, \\\label{pfwc}
  \Gamma_m^{2'} (\ov f_u) 
  &\leq C 2^{N_B} 3^m 5^{-N}, \\\label{pfwb}
  \Gamma_m^b (\ov f_u) 
  &\leq C 2^{N_B} \Big( \frac35 \Big)^M,
\end{align}
which is sufficient since $m < M < N$.

The proof of \eqref{pfwd} is similar to that in the 1-block estimate; we focus on the small differences.
Again, by the convexity of carr\'e du champ operators, we get, writing $\ov \psi^2 := \ov f_u$ and recalling the extension of $\eta_\omega : V_N^\omega \to \{0,1\}$ to $V_m^0$, 
\begin{align*}
  \Gamma_m^2 (\ov f_u)
  &\leq \sum_{\tilde \eta}  \Gamma_m^2 \big( \ov f(\cdot, \cdot, \tilde \eta) \big) \nu_\rho(\tilde \eta) \\
  &= \frac12 \sum_{\tilde \eta} \sum_{\xi, \xi' \in \Omega_m^0} \sum_{x \in V_m^0} \sum_{ y \sim x } \big[ \ov \psi(\xi^{xy}, \xi', \tilde \eta) - \ov \psi(\xi, \xi', \tilde \eta) \big]^2 \nu_\rho^2(\xi, \xi') \nu_\rho(\tilde \eta) \\
  &= \frac12 \sum_{ \eta \in \Omega_N^I } \sum_{x \in V_m^0} \sum_{ y \sim x } \big[ \ov \psi(\eta^{x_{co}y_{co}}) - \ov \psi(\eta) \big]^2 \nu_\rho( \eta)
  =: \Gamma_N^{co}(\ov f),
\end{align*} 
where we recall that $x_\omega := \varphi_\omega(x) \in V_N^0$. Note that $\Gamma_N^{co}$ is the same as $\Gamma_N^o$ in \eqref{pfvj}. Since also $\ov f$ is the same, we obtain \eqref{pfwd} from the same estimates in \eqref{pfvi}.  

The proof of \eqref{pfwc} is similar; we omit its proof.
The proof of \eqref{pfwb} relies crucially on the moving particle lemma. 
First, similar to $\Gamma_m^2$, we obtain
\begin{align} \notag 
  \Gamma_m^b (\ov f_u)
  &\leq \sum_{\tilde \eta}  \Gamma_m^b \big( \ov f(\cdot, \cdot, \tilde \eta) \big) \nu_\rho(\tilde \eta) \\\notag
  &= \frac12 \sum_{\tilde \eta} \sum_{\xi, \xi' \in \Omega_m^0} \big[ \ov \psi \big( (\xi, \xi')^{bb'}, \tilde \eta) - \ov \psi ( \xi, \xi', \tilde \eta ) \big]^2 \nu_\rho^2(\xi, \xi') \nu_\rho(\tilde \eta) \\\notag
  &= \frac12 \sum_{ \eta \in \Omega_N^I } \big[ \ov \psi(\eta^{b_{co}b_{cu}}) - \ov \psi(\eta) \big]^2 \nu_\rho( \eta)
  =: \Gamma_m^{b,u}(\ov f),
\end{align}  
where $\Gamma_m^{b,u}$ is the Dirichlet form on functions on $\Omega_N^I$ of the Kawasaki dynamics where particles jump only between the sites $b_{co}$ and $b_{cu}$ on $V_N^I$. Note, however, that $b_{co} \nsim b_{cu}$, i.e.\ they are not connected by a single edge. From a similar computation as in \eqref{pfvi}, we get
\begin{align} \notag
  \Gamma_m^{b,u}(\ov f) 
  &\leq \frac1{3^k |\cW_M^I|} \sum_{w \in \cW_M^I} \sum_{v \in \cW_k} \Gamma_m^{b,u} ( \sigma_{(-w)(-v)} f ) \\\notag
  &\leq \frac2{3^{N-m}} \sum_{w \in \cW_M^I} \sum_{v \in \cW_k} \Gamma_m^{b,u} ( \sigma_{(-w)(-v)} f ) \\\notag
  &= \frac1{3^{N-m}}  \sum_{\eta \in \Omega_N^I} \sum_{w \in \cW_M^I} \sum_{v \in \cW_k} \big[ \psi \big( \eta^{b_{wv} b_{w(v+u)}} \big) - \psi(\eta) \big]^2 \nu_\rho(\eta). 
\end{align}
We continue by applying the moving particle lemma (see Lemma \ref{l:mov:part}).
Together with the bound \eqref{pfux} (which applies since $b_{wv}, b_{w(v+u)} \in V_N^w$ for each $w \in \cW_M^I$ and each $v \in \cW_k$), we obtain
\begin{equation*}
  \Gamma_m^{b,u}(\ov f) 
  \leq \frac C{3^{N-m}} \Big( \frac53 \Big)^{N-M} \sum_{w \in \cW_M^I} \sum_{v \in \cW_k} \Gamma_N^I(f)
  \leq C' \Big( \frac53 \Big)^{N-M} 2^{N_B} \Big( \frac35 \Big)^N
  = C' 2^{N_B} \Big( \frac35 \Big)^M.
\end{equation*}
This completes the proof of \eqref{pfwb}, and therefore that of the claim in Step 1.
\medskip

The remaining three steps are similar to the 1-block estimate; we only focus on the modifications of the proofs.
\medskip

\noindent
\textbf{Step 2}: Taking $M,N \to \infty$. \eqref{pfxt} holds if
\begin{equation} \label{pfwr} 
  \limsup_{m\to\infty} \sup_{\Gamma_m^{2,b}(f) = 0} E_{\nu_\rho^2} \big[ f(\xi, \xi') |\ov \xi - \ov \xi'| \big]
  = 0,
\end{equation}
where $f$ is a density with respect to $\nu_\rho^2$. 
\medskip

To prove this, we apply the proof of Step 2 of the 1-block estimate twice: first for $N$, which yields that \eqref{pfxt} holds if 
\begin{equation} \label{pfvh}
  \limsup_{N_B,m,M} \sup_{\Gamma_m^{2,b}(f) \leq C 2^{N_B} (3/5)^M} E_{\nu_\rho^2} \big[ f(\xi, \xi') |\ov \xi - \ov \xi'| \big] 
  = 0, 
\end{equation}
and second for $M$, which yields that \eqref{pfvh} implies \eqref{pfwr}.
\medskip

\noindent
\textbf{Step 3}: Fixing the total number of particles. \eqref{pfwr} holds if
\begin{equation} \label{pfwq}
  \limsup_{m\to\infty} \max_{0 \leq j \leq 2 |V_m^0|}  
  E_{\nu^j} \big[ \big| \ov \eta - \ov \eta' \big| \big]
  = 0,
\end{equation}
where $\nu^j$ is the restriction of $\nu_\rho^2$ to 
\[
  \Omega_m^{0,2,j} := \Big\{ (\xi, \xi') \in \Omega_m^0 \times \Omega_m^0 : \sum_{x \in V_m^0} (\xi_x + \xi_x') = j \Big\}.
\]

To prove this, we note that $\Gamma_m^{2,b}(f) = 0$ means again -- but now because of the additional term $\Gamma_m^b$ -- that $f$ is constant on $\Omega_m^{0,2,j}$ for each $j$. The remainder of the proof is the same as in Step 3 of the 1-block estimate.
\medskip

\noindent
\textbf{Step 4}: Equivalence of ensembles. \eqref{pfwq} holds.
\medskip

To prove this, a few steps from the proof of Step 4 in the 1-block estimate are sufficient. Indeed, after replacing $\nu^j$ by $\nu_\alpha$ as done in \eqref{pfvg}, we replace both $\ov \eta$ and $\ov \eta'$ by $\alpha$.

\section{The hydrodynamic limit for $b \geq \frac53$}
\label{s:b}

In this section, we establish the counterpart of the hydrodynamic limit (Theorem \ref{t}) for the cases $b = \frac53$ and $b > \frac53$. The main result is stated in Theorem \ref{t:R} below.

\subsection{The reaction--diffusion equation for $b \geq \frac53$}
\label{s:HDL:Eqn:HDL:RBC}

The reaction--diffusion equation for $b = \frac53$ is given  by
\begin{equation} \label{HDL:R}
  \left\{ \begin{aligned}
    \partial_t \rho(t,x) &= \frac23 \Delta \rho(t,x) + \Phi(\rho(t,x))
    &&t \in (0,T], \ x \in K \setminus V_0  \\
    \partial^\perp \rho(t, a) &= -r(a) ( \rho(t, a) - \rho_B(a) )
    &&t \in [0,T], \ a \in V_0 \\
    \rho(0, x) &= \rho_\circ(x)
    &&x \in K,
  \end{aligned} \right.
\end{equation}
where, in addition to the data $T, \rho_B, \rho_\circ, \Phi$, a new vector $r : V_0 \to [0,\infty)$ is also given. The only difference with the equation \eqref{HDL} for $b < \frac53$ is that the Dirichlet boundary condition is replaced by a Robin boundary condition. 

The equation for $b > \frac53$ is given by \eqref{HDL:R} with the special choice $r = 0$. This turns the boundary condition into a homogeneous Neumann condition. Its definition of weak solutions and the uniqueness thereof are therefore covered by the treatment of the case $b = \frac53$. 

The definition of weak solutions (Definition \ref{d:wSol}) changes to:

\begin{defn}[Weak solution for $b \geq \frac53$] \label{d:wSol:R}
A measurable function $\rho:[0,T]\times K\to\R$ is a weak solution of \eqref{HDL:R} if
\begin{enumerate}
  \item $\rho \in L^2(\cF)$, and
  \item For all $t \in (0,T)$ and all $F \in \cC_T := C^1((0,T); \cD(\Delta)) \cap C([0,T]; \cD(\Delta))$ 
  \begin{multline} \label{weak-form:R}
    0
    = \Theta_t^r(\rho, F) 
    := \int_K \rho_t F_t \, dm 
      - \int_K \rho_\circ F_0 \, dm
      - \int_0^t \int_K \rho_s \Big( \frac23 \Delta + \partial_s \Big) F_s \, dm ds 
      \\
      - \int_0^t \int_K \Phi(\rho_s) F_s \, dm ds
      + \frac23 \int_0^t \sum_{a \in V_0} \big( \partial^\perp F_s(a) \rho_B(a) + r(a) (\rho_s(a) - \rho_B(a)) F_s(a) \big) \, ds.
  \end{multline}
\end{enumerate}
\end{defn}

The differences with Definition \ref{d:wSol} for $b < \frac53$ are as follows:
\begin{itemize}
  \item no explicit condition on $\rho_t |_{V_0}$ is given. Instead, $\rho_t(a)$ appears in an additional term in \eqref{weak-form:R}.
  \item the space of the test functions $\cC_T$ is extended to those that need not be $0$ at $V_0$.
\end{itemize}

\begin{thm} \label{t:HDL:R:un}
Weak solutions of \eqref{HDL:R} are unique for any $T, \rho_B, \rho_\circ, \Phi, r$ in the setting above (recall Section \ref{s:HDL:Eqn:HDL}). 
\end{thm}

\begin{proof} In \cite[Section 7.2]{CG}  the proof for $\Phi = 0$ is given. It is a small modification of the proof in the case $b < \frac53$, which we have briefly recalled in the proof of Theorem \ref{t:HDL:un}. Essentially, the only change is in the inner product $\cE_1$ on $\cF$ (recall \eqref{cE1}), which in the current case $b \ge \frac53$ has an additional term at $V_0$ to account for the added boundary term in \eqref{weak-form:R}. Then, also for $b \ge \frac53$, our extension of the proof to $\Phi \neq 0$ in the case $b < \frac53$ applies verbatim.
\end{proof}

\subsection{The hydrodynamic limit}

The main result of Section \ref{s:b} is the following theorem. The only differences with Theorem \ref{t} are the range of $b$, the appearance of $r$, and the equation that $\rho_t$ satisfies. In any case, we state it in full for ease of reference.

\begin{thm}[Hydrodynamic limit] \label{t:R}
Let $b \geq \frac53$ and $\lambda_\pm, c_x$ be as in Section \ref{s:GK}. Let  
\[
r := \begin{cases}
  \lambda_+ + \lambda_-
  &\text{if } b = \frac53 \\
  0
  &\text{if } b > \frac53, 
\end{cases}
\]
$\rho_B := \frac{\lambda_+}{\lambda_+ + \lambda_-}$
and $\Phi : [0,1] \to \R$ be defined from $c_x$ as in \eqref{Phi}. Let $T > 0$ and $\rho_\circ $ be as in Section \ref{t:HDL:un} such that $0 \le \rho_\circ \le 1$. Then, there exists a unique weak solution $\rho_t$ of \eqref{HDL:R} with data $T, \rho_B, \rho_\circ, \Phi, r$ such that $\rho_t(x) \in [0,1]$ for all $x \in K$ and a.e.\ $t \in (0,T)$.
Moreover, for any $(\mu_N)_{N \geq 1}$ associated to $\rho_\circ$ (recall \eqref{asso-d:meas-s}), any $t \in [0,T]$, any $\e > 0$ and any $f \in C(K)$
\begin{equation*}
  \lim_{N \to \infty} \P_{\mu_N} \Big( \Big| \frac1{|V_N|} \sum_{x \in V_N} f(x) \eta_t^N(x) - \int_K f \rho_t \, dm \Big| > \e \Big) = 0.
\end{equation*}
\end{thm} 

In the remainder of Section \ref{s:b}, we prove Theorem \ref{t:R}. The proof is very similar to that of Theorem \ref{t}; we only focus on the modifications, and otherwise follow the argumentation in Sections \ref{s:pf} and \ref{s:repl} without mentioning it.

\subsection{A different choice of $\varrho$}

While in Section \ref{s:pf:varrCG} we took $\varrho$ to be harmonic, we take it here to be even simpler: a constant function. The precise value does not matter; we take
\[
  \varrho = \frac13 \sum_{a \in V_0} \rho_B(a)
\]
as the value $\rho$ in \eqref{pfty}. 

As a consequence, the boundary values of $\varrho$ may not match with $\rho_B$. This has no consequence for the validity of Proposition \ref{p:DNG:LB} on the lower bound of the Dirichlet form of the Glauber part, but the lower bounds in \eqref{pfup} and \eqref{pfuo} on the Dirichlet form of the Kawasaki and boundary parts change to (see \cite[Lemma 5.3]{CG})
\begin{align*}
  \lrang{ \psi, - \cL_N^K  \psi }_{\nu_\varrho} + \lrang{ \psi, - b^{-N} \cL_N^B  \psi }_{\nu_\varrho}
  &\geq \Gamma_N (\psi^2, \varrho) - \frac C{b^N}
\end{align*}
for some constant $C = C(\varrho, \lambda_\pm) > 0$. Then, since $b \geq \frac53$, the lower bound on the Dirichlet form $\cD_N(\psi^2, \varrho)$ in \eqref{pfun} changes to
\begin{align} \label{pfts} 
  \cD_N(\psi^2, \varrho)
  \geq \Gamma_N (\psi^2, \varrho) - C \Big( \frac35 \Big)^N,
\end{align}
i.e., \ a nonnegative term of the boundary contribution has vanished from the right-hand side. Fortunately, we do not rely on this boundary contribution in the proof of the replacement Lemma \ref{l:replacement} on the Glauber term, and thus the ramifications of this weaker lower bound are confined to the replacement lemmas on the Kawasaki and the boundary parts.

\subsection{The associated martingale}

The difference with the corresponding Section \ref{s:pf:mg} is that $F_s(a)$ need not be $0$. For the Kawasaki and boundary parts given by $\cL_N^K$ and $\cL_N^B$, the corresponding computations are done in \cite[Section 3.4]{CG}. For the Glauber part $\cL_N^G$, the value of $F_s(a)$ has no effect on its contribution to the martingale $M_t^N(F)$ or to its quadratic variation.
From these facts, we obtain that \eqref{pfzz} becomes
\begin{multline} \label{pfuk}
  5^N \cL_N \pi_s^N(F_s)
  = \frac1{|V_N|} \sum_{x \in V_N^0} \Big( \eta_s^N (x) \Delta_N F_s(x) + c_x(\eta_s^N) (1 - 2 \eta_s^N (x)) F_s(x) \Big) \\
  - \frac{3^N}{|V_N|} \sum_{a \in V_0} \Big( \eta_s^N (a) \partial_N^\perp F_s(a) + \Big( \frac{5}{3b} \Big)^N F_s(a) (\lambda_+ (a) + \lambda_- (a))  ( \eta_s^N (a) - \rho_B(a) ) \Big)
\end{multline}
and that \eqref{pfuj} becomes
\begin{align*}
  \lrang{M^N(F)}_t
  &= \int_0^t \frac{5^N}{|V_N|^2} \sum_{x \in V_N} \sum_{ \substack{x \in V_N \\ y \sim x} } \big[ \eta_s^N (x) - \eta_s^N (y) \big]^2 [ F_s(x) - F_s(y) ]^2 ds \\
  &\quad + \int_0^t \frac2{|V_N|^2} \sum_{x \in V_N^0} c_x(\eta_s^N) (1 - \eta_s^N(x)) F_s(x)^2 \, ds  \\
  &\quad + \int_0^t \Big( \frac{5}{b} \Big)^N \frac1{|V_N|^2} \sum_{a \in V_0}  \big(  \lambda_-(a) \eta_s^N (a) + \lambda_+(a) (1 - \eta_s^N (a)) \big) F_s(a)^2 \, ds. 
\end{align*}
Regarding the third term, since $b \geq \frac53$, we have $(\frac{5}{b})^N |V_N|^{-2} \leq C 3^{-N}$, and thus the previous bound in \eqref{pfui} on $\lrang{M^N(F)}_t$ remains valid.

\subsection{The limiting density $\rho$ is a weak solution}

Analogously to Section \ref{s:pf}, we obtain that $\Q_N$ is tight, and that any limit point $\Q$ is concentrated on paths $\pi_\bullet$ with density $\rho \in L^2(0,T; \cF)$. The main statement to be shown is the counterpart of \eqref{pfzw}, which is here given by
\begin{equation} \label{pftt}
   \Q \Big( \sup_{0 \leq t \leq T} |\Theta_t^r(\rho, F)| > \e \Big) = 0,
\end{equation}
where $\Theta_t^r(\rho, F)$ is the weak form in \eqref{weak-form:R}. 

The only required modification to the proof of \eqref{pftt} as given in Section \ref{s:pf:wSol} are the replacements of the boundary term under $\Q_N$. For the boundary term, note that $M_t^N(F)$ has an additional term corresponding to the last term in \eqref{pfuk} which contains $F_s(a)$. The corresponding replacement is detailed in the proof of \cite[Proposition 6.7]{CG}. The main step is the replacement lemma given by \cite[Lemma 5.5]{CG}, which requires $\varrho$ to be constant. Similar to the argument above \eqref{pfue}, even though \cite[Lemma 5.5]{CG} is stated for $\cL_N$ without the Glauber term, its proof applies to our setting if $\cD_N(\psi^2, \varrho)$ is sufficiently bounded from below. From the end of the proof of \cite[Lemma 5.5]{CG} it is clear that our lower bound in \eqref{pfts} is sufficient.
This completes the proof of Theorem \ref{t:R}.

We end with a remark on a simplification of the proof of the replacement of the Glauber term, which is made possible by the lower bound in \eqref{pfts} on $\cD_N(\psi^2, \varrho)$ in the case $b \geq \frac53$. Then, we may use the constant function $\varrho$ in the proof of Lemma \ref{l:replacement}. With this choice of $\varrho$ there is no need to work with $N_B$ and to treat separately the contributions on the thick boundary $K^B$ and the reduced interior $K^I$.

\section*{Acknowledgements}

PvM has received financial support from JSPS KAKENHI Grant Numbers JP20K14358 and JP24K06843.
KT has received financial support from JSPS KAKENHI Grant Number JP22K13929.


\begin{thebibliography}{99}

%
%


\bibitem{Chen}{\sc J.\ P.\ Chen},
{\it Local ergodicity in the exclusion process on an infinite weighted graph},
arXiv:1705.10290 (2017), 36 pp.

\bibitem{Chen2}{\sc J.\ P.\ Chen},
{\it The moving particle lemma for the exclusion process on a weighted graph},
Electron. Commun. Probab., \textbf{22}(47) (2017), 13 pp.

\bibitem{CG}{\sc J.\ P.\ Chen and P.\ Gon\'calves},
{\it Asymptotic behavior of density in the boundary-driven exclusion process on the Sierpinski gasket}.
Math. Phys. Anal. Geom., \textbf{24} (2021), 1--65. 

\bibitem{DFL}{\sc A.\ De Masi, P.\ A.\ Ferrari, and J.\ L.\ Lebowitz},
{\it reaction--diffusion equations for interacting particle systems}.
J. Stat. Phys., \textbf{44} (1986), 589--644.

%
%
%

\bibitem{FLT}{\sc J.\ Farfan, C.\ Landim and K.\ Tsunoda},
{\it Static large deviations for a reaction--diffusion model},
Probab. Theory Relat. Fields, \textbf{174} (2019), 49--101. 

%
%
%

\bibitem{vGR}{\sc B.\ van Ginkel and F.\ Redig},
{\it Hydrodynamic limit of the symmetric exclusion process on a compact Riemannian manifold}.
J. Stat. Phys., \textbf{178} (2020), 75--116.

\bibitem{Jar}{\sc M.\ Jara},
{\it Hydrodynamic limit for a zero-range process in the {S}ierpinski gasket},
Comm. Math. Phys., \textbf{288} (2009), 773--797.


\bibitem{JLS}{\sc M.\ Jara, C.\ Landim and S.\ Sethuraman},
{\it Nonequilibrium fluctuations for a tagged particle in mean-zero one-dimensional zero-range processes},
Probab. Theory Relat. Fields, \textbf{145} (2009), 565--590.

\bibitem{JLS2}{\sc M.\ Jara, C.\ Landim and S.\ Sethuraman},
{\it Nonequilibrium fluctuations for a tagged particle in one-dimensional sublinear zero-range processes}.
Ann. Inst. Henri Poincar\'e Probab. Stat., \textbf{49} (2013), 611--637. 

%
%

\bibitem{JLV}{\sc G.\ Jona-Lasinio, C.\ Landim and M.E.\ Vares},
{\it Large deviations for a reaction diffusion model},
Probab. Theory Relat. Fields, \textbf{97} (1993), 339--361. 

\bibitem{JRV}{\sc J.\ Junn\'e, F.\ Redig and R.\ Versendaal},
{\it Hydrodynamic limit of the symmetric exclusion process on complete Riemannian manifolds and principal bundles},
arXiv:2410.20167 (2024), 24 pp.

\bibitem{Kig}{\sc J.\ Kigami},
{\it Analysis on Fractals},
Cambridge University Press, Cambridge, UK, (2001).

\bibitem{KOV}{\sc C.\ Kipnis, S.\ Olla and S.R.S. \ Varadhan},
{\it Hydrodynamics and large deviation for simple exclusion processes},
Comm. Pure Appl. Math., \textbf{42} (1989), 115--137.

\bibitem{KL}{\sc C.\ Kipnis and C.\ Landim},
{\it Scaling Limits of Interacting Particle Systems},
Grundlehren der Mathematischen Wissenschaften, 
Heidelberg, \textbf{320} (2012), 491 pp.

\bibitem{LT}{\sc C.\ Landim and K.\ Tsunoda},
{\it Hydrostatics and dynamical large deviations for a reaction--diffusion model}.
Ann. Inst. Henri Poincar\'e Probab. Stat., \textbf{54} (2018), 51--74. 

%
%
%
%

\bibitem{Str}{\sc R.\ S.\ Strichartz},
{\it Differential Equations on Fractals: A Tutorial},
Princeton University Press, Princeton, NJ, (2006).

\bibitem{Tan}{\sc R.\ Tanaka},
{\it Hydrodynamic limit for weakly asymmetric simple exclusion processes in crystal lattices},
Comm. Math. Phys., \textbf{315} (2012), 603--641.

%





\end{thebibliography}
\end{document}